\documentclass[11pt,oneside]{article}

\usepackage{style_mjh}
\usepackage[numbers]{natbib}
\usepackage{amsmath}                                          
\usepackage{amsthm}                                           
\theoremstyle{definition} \newtheorem{defn}{Definition}       
\theoremstyle{plain} \newtheorem{prop}[defn]{Proposition}          
\theoremstyle{plain}                 
\theoremstyle{plain} \newtheorem{lem}[defn]{Lemma}                  
\theoremstyle{plain}              
\theoremstyle{remark} \newtheorem{rmk}[defn]{Remark}                
\theoremstyle{remark}                 

\usepackage{enumitem}
\makeatletter
\def\namedlabel#1#2{\begingroup
    #2%
    \def\@currentlabel{#2}%
    \phantomsection\label{#1}\endgroup
}
\makeatother

\usepackage[pdfauthor={Matthew James Holland},%
pdftitle={Holland (2019): Distribution-robust mean estimation via smoothed random perturbations},%
colorlinks=true,%
linkcolor=blue,%
citecolor=blue,%
pdftex]{hyperref} 

\begin{document}

\title{\textbf{Distribution-robust mean estimation via smoothed random perturbations}}
\author{
  Matthew J.~Holland\thanks{Please direct correspondence to \texttt{matthew-h@ar.sanken.osaka-u.ac.jp}.}\\
  Osaka University
}
\date{} 

\maketitle

\begin{abstract}
We consider the problem of mean estimation assuming only finite variance. We study a new class of mean estimators constructed by integrating over random noise applied to a soft-truncated empirical mean estimator. For appropriate choices of noise, we show that this can be computed in closed form, and utilizing relative entropy inequalities, these estimators enjoy deviations with exponential tails controlled by the second moment of the underlying distribution. We consider both additive and multiplicative noise, and several noise distribution families in our analysis. Furthermore, we empirically investigate the sensitivity to the mean-standard deviation ratio for numerous concrete manifestations of the estimator class of interest. Our main take-away is that an inexpensive new estimator can achieve nearly sub-Gaussian performance for a wide variety of data distributions.
\end{abstract}

\tableofcontents

\section{Introduction}\label{sec:overview}

In this work, we consider the problem of mean estimation under very weak assumptions on the underlying distribution; all we assume is that the variance is finite. This problem is important because the practitioner often will not know \textit{a priori} whether or not the underlying data distribution is sub-Gaussian in nature or whether it is heavy-tailed with infinite higher-order moments \citep{devroye2016a}. Furthermore, procedures which provide strong statistical estimation guarantees in this setting have been shown to have many applications in modern machine learning tasks, where off-sample generalization performance is measured using a risk (expected loss) function to be estimated empirically as a core feedback mechanism, with more robust feedback leading to provably stronger learning guarantees \citep{brownlees2015a,chen2017a,holland2019c}.

Given a sample $X_{1},\ldots,X_{n}$ of $n$ independent random variables taking values in $\RR$ with common distribution $\ddist$, the traditional approach for estimation of the mean $\exx_{\ddist}X$ is to use the empirical mean $\xbar \defeq n^{-1} \sum_{i=1}^{n} X_{i}$, for which many optimality properties are well-known. For example, if the data is Normally distributed with variance $\sigma^{2}$, then confidence intervals for the deviations can be computed exactly, and with probability no less than $1-2\delta$, one has
\begin{align*}
|\xbar - \exx_{\ddist}X| \leq \frac{\sigma}{\sqrt{n}} \Phi^{-1}(1-\delta)
\end{align*}
where $\Phi$ denotes the Normal cumulative distribution function. Even without Normal assumptions, since we have finite variance, asymptotically the central limit theorem tells us that the same kinds of guarantees are possible, where as $n \to \infty$ we have
\begin{align*}
\prr\left\{ |\xbar - \exx_{\ddist}X| > \frac{\sigma}{\sqrt{n}} \Phi^{-1}(1-\delta/2) \right\} \to \delta.
\end{align*}
This deviation bound is a natural benchmark against which to compare other estimators, since in the Normal case, the empirical mean $\xbar$ is essentially optimal, in the following sense \citep{catoni2012a}. Let $\PP$ be a family of distributions including all Normal distributions with some finite variance $\sigma^{2}$. Take any estimator $\xhat$, and denote any valid one-sided deviation bounds by $\varepsilon_{\delta}(\xhat,\ddist)$, where
\begin{align*}
\prr\left\{ \xhat-\exx_{\ddist}X > \varepsilon_{\delta}(\xhat,\ddist) \right\} \leq \delta
\end{align*}
for $\ddist \in \PP$ and $\delta \in (0,1)$. Then, regardless of the construction of $\xhat$, for any confidence level $\delta$, there always exists an unlucky Normal distribution $\ddist_{\text{bad}} \in \PP$ with $\vaa_{\ddist_{\text{bad}}} X = \sigma^{2}$ such that
\begin{align*}
\varepsilon_{\delta}(\xhat,\ddist_{\text{bad}}) \geq \frac{\sigma^{2}}{\sqrt{n}} \Phi^{-1}(1-\delta).
\end{align*}
Analogous statements hold for the lower tail, meaning that the benchmark set by the empirical mean in the case of Normal data is essentially the best we can expect of any estimator, in terms of dependence on $n$ and $\delta$ and guarantees that hold uniformly over the model $\PP$.

With this natural benchmark in mind, \citet{devroye2016a} did an in-depth study of so-called \textit{sub-Gaussian estimators}, namely any estimator which satisfies
\begin{align}\label{eqn:subG_estimators}
\prr\left\{ |\xhat-\exx_{\ddist}X| > c \, \sigma_{\ddist} \sqrt{\frac{(1+\log(\delta^{-1}))}{n}} \right\}, \quad \ddist \in \PP
\end{align}
where $\PP$ is typically a large, non-parametric class of distributions, $\sigma_{\ddist}^{2} \defeq \vaa_{\ddist}X$ is the variance of $\ddist$, and $c>0$ is some distribution-free constant. Since $\Phi^{-1}(1-\delta/2) \leq \sqrt{2\log(2\delta^{-1})}$, this is a slight weakening of the benchmark given above, but fundamentally captures the same phenomena. One important fact is that assuming only finite variance, namely the model $\PP_{2} \defeq \{\ddist: \sigma_{\ddist}^{2} < \infty\}$, the empirical mean $\xbar$ is \textit{not} a sub-Gaussian estimator. As a salient example, \citet{catoni2012a} shows that one can always create a distribution $\ddist_{\text{bad}} \in \PP_{2}$ such that with probability \textit{at least} $\delta$, one has
\begin{align*}
|\xbar-\exx_{\ddist_{\text{bad}}}| \geq \frac{\sigma_{\ddist_{\text{bad}}}}{\sqrt{n\delta}} \left(1 - \frac{e\delta}{n}\right)^{(n-1)/2},
\end{align*}
namely a lower bound which says that the guarantees provided by Chebyshev's inequality, with polynomial, rather than logarithmic dependence on $1/\delta$, are essentially tight for the finite-variance model $\PP_{2}$.

Given that the empirical mean $\xbar$ is not distribution-robust in the sub-Gaussian sense just stated, it is natural to ask whether, under such weak assumptions, there exist sub-Gaussian estimators at all. The answer in this case is affirmative, although the construction of $\xhat$ depends on the desired confidence level $\delta$. This is indeed necessary, as \citet{devroye2016a} prove that for $\PP_{2}$ one cannot construct a sub-Gaussian estimator whose design is free of $\delta$. One of the most lucid examples is the median-of-means estimator \citep{lerasle2011a}, which partitions $\{1,\ldots,n\}$ into $k$ disjoint subsets, computes a sample mean on each $\xbar_{1},\ldots,\xbar_{k}$, and finally returns the median $\xhat_{\text{MOM}} = \med\{\xbar_{1},\ldots,\xbar_{k}\}$. With the right number of partitions (e.g., $k = \lceil\log(\delta^{-1})\rceil$), the estimator $\xhat_{\text{MOM}}$ is sub-Gaussian. Another lucid example is that of M-estimators with variance-controlled scaling \citep{catoni2012a}, taking the form
\begin{align*}
\xhat_{\text{M}} = \argmin_{\theta \in \RR} \sum_{i=1}^{n} f\left(\frac{\theta-X_{i}}{s_{\delta,\ddist}}\right)
\end{align*}
where $f$ is an appropriate convex function, and the parameter $s_{\delta,\ddist}^{2}$ scales as $O(n\sigma_{\ddist}^{2}/\log(\delta^{-1}))$. Furthermore, one of the main results of \citet{devroye2016a} is a novel estimator construction technique which is provably sub-Gaussian with nearly optimal constants assuming only finite variance, but their procedure is not computationally tractable.

\paragraph{Our contributions}
All known sub-Gaussian estimators require solving some sub-problem whose solution is implicitly defined; even the most practical choices such as the median-of-means and M-estimator approaches amount to minimizing a data-dependent convex function. In this work, we consider a new class of estimators which can be computed directly and precisely, without any iterative sub-routines. The cost for this is we give up sub-Gaussianity; we show that the estimators of interest satisfy $1-2\delta$ deviation bounds of the form
\begin{align*}
|\xhat-\exx_{\ddist}X| \leq c \sqrt{\frac{\exx_{\ddist}X^{2}(1+\log(\delta^{-1}))}{n}}, \quad \ddist \in \PP_{2}
\end{align*}
for an appropriate constant $c>0$ which only depends on $\xhat$ and $\PP_{2}$. The basic form is the same as the sub-Gaussian condition of \citet{devroye2016a}, but because the second moment controls the bound instead of the variance, it becomes sensitive to the absolute value of the mean $\exx_{\ddist}X$, though as we demonstrate shortly, a simple sample-splitting technique works well to mitigate this sensitivity both in theory and in practice. The estimators we construct are built by first considering the application of multiplicative or additive noise to a soft-truncated version of the empirical mean, and then smoothing out these effects by taking expectation with respect to these random perturbations, whose distribution we control. We consider adding noise from Bernoulli, Normal, Weibull, and Student-t families to provide some concrete examples of the estimator class of interest. In particular, the Bernoulli variety is computationally simplest and performs best in our experimental setting, offering a new alternative to the sub-Gaussian mean estimators cited earlier, with the appeal of no computational error and more transparent analysis.

The rest of the paper is structured as follows. In section \ref{sec:estimator} we introduce the estimator class of interest, and highlight some basic statistical and computational principles which will be of use for subsequent analysis. Our main theoretical results are organized in section \ref{sec:theory}, where we prove statistical error bounds and derive closed-form expressions for concrete examples from the class of interest. In section \ref{sec:empirical} we conduct a series of controlled experiments in which we analyze performance, evaluating in particular the sensitivity to the size of the mean relative to the standard deviation. A brief discussion and concluding remarks close the paper in section \ref{sec:conclusions}.

\section{Estimator class of interest}\label{sec:estimator}

Our aim is to study the behavior of a class of new estimators which use moment-dependent scaling, smoothed noise (both additive and multiplicative), and bounded soft truncation. Let $s>0$ denote a generic scaling parameter to be determined shortly, and let $\epsilon_{1},\ldots,\epsilon_{n}$ denote $n$ independent copies of a noise random variable $\epsilon \sim \ndist$. Both $s$ and $\ndist$ are assumed to be under our control. For simplicity, we focus on the following two main types of estimators:
\begin{enumerate}
\item \textbf{Multiplicative:} Assuming $\exx_{\ndist} \epsilon \neq 0$, construct estimator as
\begin{align}\label{eqn:xm_defn}
\xm \defeq \exx \left(\frac{s}{n} \sum_{i=1}^{n} \trunc\left(\frac{X_{i}\epsilon_{i}}{s}\right)\right).
\end{align}

\item \textbf{Additive:} Assuming $\exx_{\ndist} \epsilon = 0$, construct estimator as
\begin{align}\label{eqn:xa_defn} 
\xa \defeq \exx \left(\frac{s}{n} \sum_{i=1}^{n} \trunc\left(\frac{X_{i}+\epsilon_{i}}{s}\right)\right).
\end{align}
\end{enumerate}
First, regarding the soft truncation function $\trunc$, any differentiable, odd, non-decreasing function is plausible, but for concreteness we use the convenient sigmoid function of \citet{catoni2017a}, given in (\ref{eqn:trunc_defn}) and pictured in Figure \ref{fig:trunc_catgiu}.
\begin{align}\label{eqn:trunc_defn}
\trunc(u) =
\begin{cases}
u - u^{3}/6, & -\sqrt{2} \leq u \leq \sqrt{2}\\
2\sqrt{2}/3, & u > \sqrt{2}\\
-2\sqrt{2}/3, & u < -\sqrt{2}
\end{cases}
\end{align}
Second, the expectation is taken with respect to the product measure induced by the sample $\epsilon_{1},\ldots,\epsilon_{n}$. As we shall see in section \ref{sec:theory}, with proper choice of $\ndist$, using the convenient polynomial form of $\trunc$, we can often compute $\xm$ and $\xa$ directly.

\begin{figure}[t]
\centering
\includegraphics[width=0.5\textwidth]{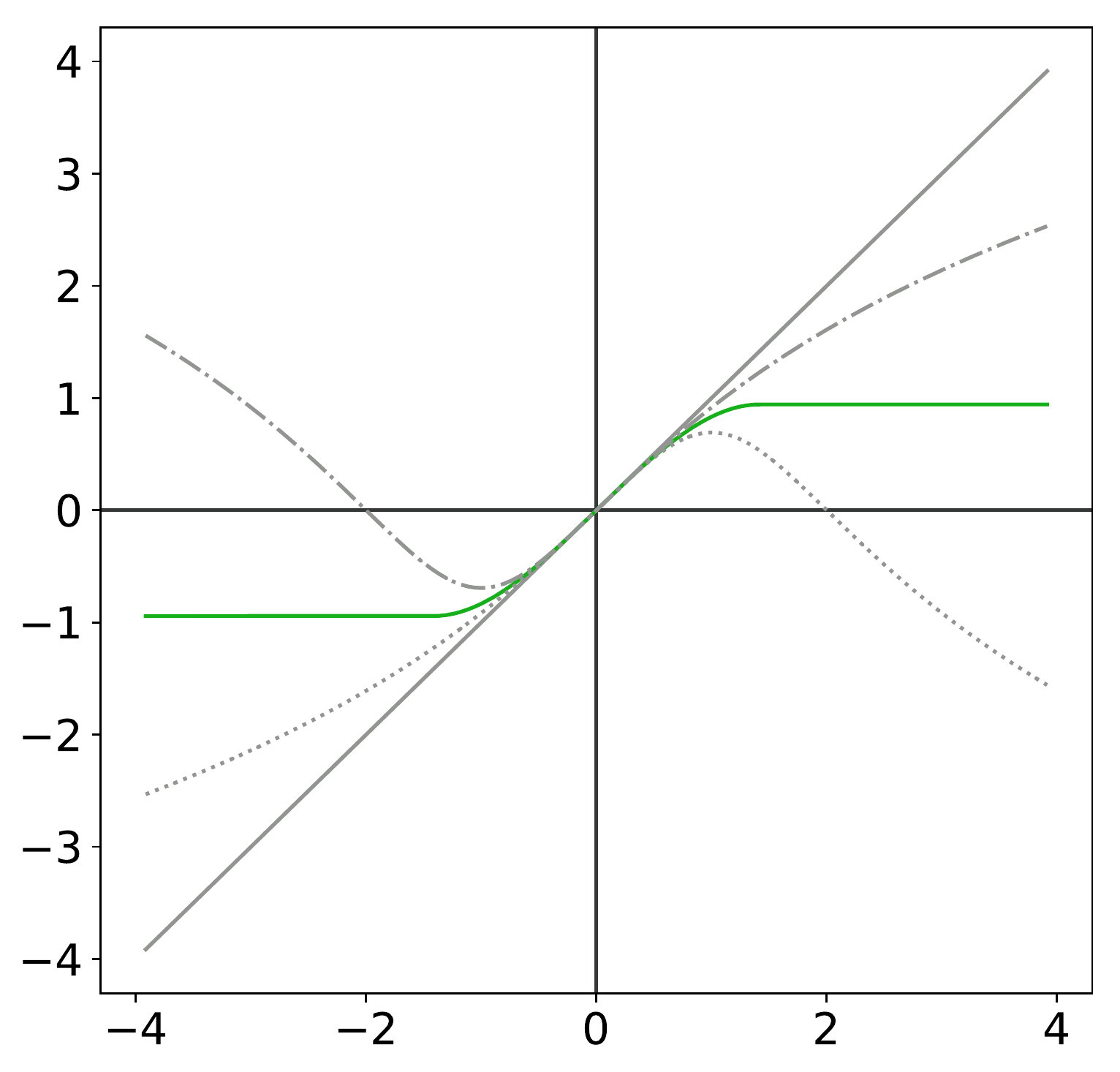}
\caption{Graph of $\trunc(u)$ (in green), along with upper and lower bounds given in (\ref{eqn:trunc_keyprop}).}
\label{fig:trunc_catgiu}
\end{figure}

\subsection{Computation of new estimators}

Here we consider some general principles which will aid us in computing the estimators $\xm$ and $\xa$ just introduced. To begin, note that the piecewise function $\trunc$ can be written explicitly using indicator functions as
\begin{align}\label{eqn:trunc_piecewise}
\trunc(u) = \left(u - \frac{u^{3}}{6}\right)\left( I\{u \leq \sqrt{2}\} - I\{u < -\sqrt{2}\} \right) + \frac{2\sqrt{2}}{3}\left(1 - I\{u \leq \sqrt{2}\} - I\{u < -\sqrt{2}\} \right).
\end{align}
Let $W$ denote an arbitrary random variable. We consider computation of $\exx\trunc(a+bW)$, where expectation is taken with respect to $W$, and $a \in \RR$ and $b > 0$ are respectively shift and scale parameters. To streamline implementation, for integer $k>0$ and input $u \in \RR$, we introduce the notation
\begin{align}
\label{eqn:comp_Mk}
M_{a,b}^{k}(u) & \defeq \exx W^{k} I\left\{ a+bW \leq u \right\},\\
\label{eqn:comp_Dk}
D_{a,b}^{k}(u) & \defeq M_{a,b}^{k}(u) - M_{a,b}^{k}(-u).
\end{align}
Some care is required with the value of $b$. For the case of $b=0$, it follows immediately that $\exx \trunc(a+bW) = \trunc(a)$. When $b \neq 0$, then we need to pay attention to the sign,\footnote{As an obvious example, say $W \sim \text{Normal}(0,1)$. Unless $u=a$, taking expectation over $(-\infty,(u-a)/b]$ and $[(u-a)/b,\infty)$ will respectively lead to different results.} computing as
\begin{align}
\label{eqn:comp_Mk_signs}
M_{a,b}^{k}(u) =
\begin{cases}
\displaystyle M_{0,1}^{k}\left( \frac{u-a}{b} \right), & \text{ if } b>0 \medskip\\
\displaystyle \exx W^{k} - M_{0,1}^{k}\left( \frac{u-a}{b} \right), & \text{ if } b<0.
\end{cases}
\end{align}
This equality follows from straightforward rearrangements. When $b>0$, note
\begin{align*}
I\{a+bW \leq u\} = I\{W \leq (u-a)/b\}
\end{align*}
for any choice of $u, a \in \RR$. When $b<0$, note that
\begin{align*}
I\{a+bW \leq u\} = I\{W \geq (u-a)/b\} = 1 - I\{W < (u-a)/b\}
\end{align*}
from which the second case in (\ref{eqn:comp_Mk_signs}) is obtained. Assuming then that evaluating $M_{0,1}^{k}(u)$ is tractable, obtaining $\exx\trunc(a+bW)$ can be reduced to direct computations as
\begin{align}
\label{eqn:comp_breakdown}
F_{W}(a,b) \defeq \exx \trunc\left(a+bW\right) = F_{0}(a,b) - F_{3}(a,b)
\end{align}
where we have
\begin{align*}
F_{0}(a,b) & \defeq \frac{2\sqrt{2}}{3} + M_{a,b}^{0}(\sqrt{2})\left(a-\frac{2\sqrt{2}}{3}-\frac{a^{3}}{6}\right) - M_{a,b}^{0}(-\sqrt{2})\left(a+\frac{2\sqrt{2}}{3}-\frac{a^{3}}{6}\right)\\
F_{3}(a,b) & \defeq \left(\frac{a^{2}}{2}-1\right)b D_{a,b}^{1}(\sqrt{2}) + \frac{ab^{2}}{2} D_{a,b}^{2}(\sqrt{2}) + \frac{b^{3}}{6} D_{a,b}^{3}(\sqrt{2}).
\end{align*}
The above follows from straightforward algebra, using the convenient form (\ref{eqn:trunc_piecewise}). In practice then, we will need to evaluate $M_{0,1}^{k}(\cdot)$ for degrees $k=0,1,2,3$. Since the actual computations depend completely on the distribution of $W$ here, detailed derivations for different distribution families will be given in section \ref{sec:theory}.

\subsection{Two general-purpose deviation bounds}

Recall the setting described above, in which we have iid data $X_{1},\ldots,X_{n}$ and iid ``strategic noise'' $\epsilon_{1},\ldots,\epsilon_{n}$, respectively distributed as $X \sim \ddist$ and $\epsilon \sim \ndist$, inspired by the approach of \citet{catoni2017a}, we may obtain convenient inequalities depending on the noise distribution $\ndist$, which is ``posterior'' in the sense that it may depend on the sample, and a ``prior'' distribution $\prior$, which must be specified in advance. The fundamental underlying property we use is illustrated in Lemma \ref{lem:cat17type_KLbound} below.\footnote{See \citet{holland2019pb} for an elementary proof.}

\begin{lem}\label{lem:cat17type_KLbound}
Fix an arbitrary prior distribution $\prior$ on $\RR$, and consider $f:\RR^{2} \to \RR$, assumed to be bounded and measurable. It follows that with probability no less than $1-\delta$ over the random draw of the sample, we have
\begin{align*}
\exx \left(\frac{1}{n} \sum_{i=1}^{n} f(X_{i},\epsilon_{i})\right) \leq \int \log\exx_{\ddist}\exp(f(x,\epsilon)) \, d\ndist(\epsilon) + \frac{\KL(\ndist;\prior) + \log(\delta^{-1})}{n},
\end{align*}
uniform in the choice of $\ndist$, where expectation on the left-hand side is over the noise sample.
\end{lem}
\noindent In the context of $\xm$ and $\xa$ defined in section \ref{sec:estimator}, we can obtain convenient upper bounds through special cases of Lemma \ref{lem:cat17type_KLbound}. The key property of the truncation function (\ref{eqn:trunc_defn}) is that over its entire domain, the following upper and lower bounds hold (see Figure \ref{fig:trunc_catgiu} for an illustration):
\begin{align}\label{eqn:trunc_keyprop}
-\log\left(1 - u + u^{2}/2\right) \leq \trunc(u) \leq \log\left(1 + u + u^{2}/2\right), \qquad u \in \RR.
\end{align}
As such, when we set $f(x,\epsilon) = \trunc(x\epsilon/s)$ for the case of $\xm$, and $f(x,\epsilon) = \trunc((x+\epsilon)/s)$ for the case of $\xa$, then using the bounds (\ref{eqn:trunc_keyprop}), it follows that each of these estimators enjoy the following upper bounds, each of which holds with probability at least $(1-\delta)$, uniform in the choice of noise distribution $\ndist$:
\begin{align}
\label{eqn:xm_KL_bound}
\frac{\xm}{s} & \leq \int \left( \frac{\epsilon \exx_{\ddist}X}{s} + \frac{\epsilon^{2}\exx_{\ddist}X^{2}}{2s^{2}} \right) \, d\ndist(\epsilon) + \frac{\KL(\ndist;\prior) + \log(\delta^{-1})}{n}\\
\label{eqn:xa_KL_bound}
\frac{\xa}{s} & \leq \int \left( \frac{\exx_{\ddist}X + \epsilon}{s} + \frac{\exx_{\ddist}(x+\epsilon)^{2}}{2s^{2}} \right) \, d\ndist(\epsilon) + \frac{\KL(\ndist;\prior) + \log(\delta^{-1})}{n}.
\end{align}
We emphasize that the randomness in the estimators $\xm$ and $\xa$ is exclusively due to the data sample, since we are integrating out any randomness due to the noise. The reason we restrict the definition of $\xm$ to the case of noise with non-zero mean is purely because we are interested in deviation bounds. For any noise distribution with $\exx_{\ndist}\epsilon = 0$ and $\exx_{\ndist}\epsilon^{2} = 1$, the first term on the right-hand side of (\ref{eqn:xm_KL_bound}) evaluates to
\begin{align*}
\int \left( \frac{\epsilon \exx_{\ddist}X}{s} + \frac{\epsilon^{2}\exx_{\ddist}X^{2}}{2s^{2}} \right) \, d\ndist(\epsilon) = \frac{\exx_{\ddist}X^{2}}{2s^{2}},
\end{align*}
meaning that the bounds on $\xhat$ in this case are always free of $\exx_{\ddist}X$, which spoils this approach for seeking bounds on $|\xm-\exx_{\ddist}X|$. On the other hand, the reason for restricting $\exx_{\ndist} \epsilon = 0$ in the definition of $\xa$ is to prevent an artefact in the upper bound on deviations that does not depend on $s$.

\section{Theoretical analysis}\label{sec:theory}

Here we consider a number of different noise distribution families for constructing the estimators (\ref{eqn:xm_defn}) and (\ref{eqn:xa_defn}) introduced in the previous section. For each distribution, we seek both statistical error guarantees, as well as explicit forms for efficiently computing the estimators.

\subsection{Bernoulli noise}\label{sec:bydistro_bernoulli}

Perhaps the simplest choice of noise distribution is that of randomly deleting observations with a fixed probability, namely the case of Bernoulli noise $\epsilon \in \{0,1\}$. The following result makes this concrete.

\begin{prop}[Deviation bounds and estimator computation]\label{prop:bernoulli_all}
Consider noise $\ndist = \text{Bernoulli}(\theta)$ for some $\theta \in (0,1)$, and prior $\prior = \text{Bernoulli}(1/2)$. The estimator $\xm$ in (\ref{eqn:xm_defn}) takes the form
\begin{align*}
\xm = \frac{\theta s}{n} \sum_{i=1}^{n} \trunc\left(\frac{X_{i}}{s}\right)
\end{align*}
and satisfies the following deviation bound,
\begin{align*}
|\xm/\theta - \exx_{\ddist}X| \leq \frac{\exx_{\ddist}X^{2}}{2s} + \frac{s}{n\theta}\left( \theta\log 2\theta + (1-\theta)\log 2(1-\theta) + \log(\delta^{-1}) \right)
\end{align*}
with probability no less than $1-2\delta$.
\end{prop}

\begin{rmk}[Centered estimates]
To reduce the dependence of the above estimator on the second moments of the underlying distribution, use of an ancillary mean estimator to approximately center the data is a useful strategy. For example, from the full sample of $X_{1},\ldots,X_{n}$, let the first $m < n$ observations be used to construct such an ancillary estimator, denoted
\begin{align*}
\xbar_{\trunc} = \frac{s}{m} \sum_{i=1}^{m} \trunc\left(\frac{X_{i}}{s}\right)
\end{align*}
where $s^{2} = m \exx_{\ddist}X^{2} / 2\log(\delta^{-1})$. Then shift the remaining data points from $X_{i} \mapsto x^{\prime}_{i}$ as $x^{\prime}_{i} \defeq X_{i} - \xbar_{\trunc}$, for each $i=m+1,\ldots,n$. Writing the upper bound in Proposition \ref{prop:bernoulli_all} depending on the full $n$-sized sample as $\varepsilon_{n}$, the second moment of the shifted data can be readily bounded above by $\exx_{\ddist}(x^{\prime})^{2} \leq \vaa_{\ddist}x + \varepsilon^{2}_{m}$. Using this upper bound to scale $\xm$, this time applied to the centered dataset $\{x_{m+1}^{\prime},\ldots,X_{n}^{\prime}\}$ of size $n-m$, write $\xhat^{\prime}$ for the resulting estimator. Shifting this back into the original position, we have our final output, namely $\xhat = \xhat^{\prime} + \xbar_{\trunc}$, which enjoys variance-dependent deviation tails of the form
\begin{align*}
\prr\{|\xhat - \exx_{\ddist}X| > \epsilon \} \leq 4\exp\left(\frac{-(n-m)\epsilon^{2}}{2(\vaa_{\ddist}x + \varepsilon^{2}_{m})}\right).
\end{align*}
\end{rmk}

\subsection{Normal noise}\label{sec:bydistro_normal}

\begin{prop}[Deviation bounds]\label{prop:normal_devbd}
Additive case: consider noise $\ndist = \text{Normal}(0,\beta^{-1})$ and prior $\prior = \text{Normal}(1,\beta^{-1})$, setting $\beta^{2} = n/s^{2}$ and $s^{2} = n\exx_{\ddist}X^{2}/(2\log(\delta^{-1}))$. Then, we have
\begin{align*}
|\xa - \exx_{\ddist}X| \leq \sqrt{\frac{2\exx_{\ddist}X^{2}\log(\delta^{-1})}{n}} + \frac{1}{\sqrt{n}}
\end{align*}
with probability no less than $1-2\delta$ over the draw of the sample.

Multiplicative case: consider noise $\ndist = \text{Normal}(1,\beta^{-1})$ and prior $\prior = \text{Normal}(2,\beta^{-1})$, setting $\beta^{2} = n \exx_{\ddist}X^{2} / s^{2}$ and $s^{2} = n \exx_{\ddist}X^{2} / 2\log(\delta^{-1})$. Then, we have
\begin{align*}
|\xm - \exx_{\ddist}X| \leq \sqrt{\frac{2\exx_{\ddist}X^{2}\log(\delta^{-1})}{n}} + \sqrt{\frac{\exx_{\ddist}X^{2}}{n}}
\end{align*}
with probability no less than $1-2\delta$ over the draw of the sample.
\end{prop}

\begin{rmk}[Comparison of additive and multiplicative noise]
The key difference in terms of the performance guarantees available for $\xm$ and $\xa$ under the Normal noise setting is that when $\exx_{\ddist}X^{2} > 1$, the additive case has tighter bounds, and when $\exx_{\ddist}X^{2} < 1$, the multiplicative case has tighter bounds. A much smaller technical difference is that we have $1-\delta$ confidence intervals for the multiplicative case, but $1-\delta$ intervals for the additive case.
\end{rmk}

Next, let us consider computation of the estimators under Normal noise. 

\begin{prop}[Estimator computation, \citep{catoni2017a}]\label{prop:normal_comp}
For $W \sim \text{Normal}(0,1)$, the key quantities are computed as
\begin{align*}
M_{0,1}^{0}(u) & = \Phi(u)\\
M_{0,1}^{1}(u) & = \frac{-1}{\sqrt{2\pi}}\exp\left(\frac{-u^{2}}{2}\right)\\
M_{0,1}^{2}(u) & = M_{0,1}^{0}(u) + u M_{0,1}^{1}(u)\\
M_{0,1}^{3}(u) & = (u^{2}+2) M_{0,1}^{1}(u)
\end{align*}
where $\Phi$ denotes the standard Normal CDF.
\end{prop}
\noindent With Proposition \ref{prop:normal_comp} in hand, computation is very straightforward using the general form
\begin{align*}
\xhat & = \frac{s}{n} \sum_{i=1}^{n} F_{W}(a_{i},b_{i}),
\end{align*}
recalling the definition of $F_{W}$ in (\ref{eqn:comp_breakdown}). To implement the special case for which the bounds of Proposition \ref{prop:normal_devbd} hold, this amounts to
\begin{align*}
\xm & \text{ case: } \enspace a_{i}=\frac{X_{i}}{s}, \enspace b_{i} = \frac{|X_{i}|}{s\sqrt{\beta}}\\
\xa & \text{ case: } \enspace a_{i}=\frac{X_{i}}{s}, \enspace b_{i} = \frac{1}{s\sqrt{\beta}}.
\end{align*}

\begin{rmk}[Convenient computation]
For the case of $\xm$, note that using the value of $b_{i} = |X_{i}| / s\sqrt{\beta}$ rather than the signed $X_{i} / s\sqrt{\beta}$ is computationally convenient because then we only need to consider the positive case in evaluating (\ref{eqn:comp_Mk_signs}). Of course, the quantities being computed in either case are equivalent. To see this, just note that since $\ndist = \text{Normal}(1,\beta^{-1})$ and $W \sim \text{Normal}(0,1)$ we have conditioned on any $X_{i} \in \RR$,
\begin{align*}
\frac{X_{i}}{s} \epsilon = \frac{X_{i}}{s}\left(1+\text{Normal}(0,\beta^{-1})\right) = \frac{X_{i}}{s} + \frac{X_{i}}{s\sqrt{\beta}} W = \frac{X_{i}}{s} + \frac{|X_{i}|}{s\sqrt{\beta}} W
\end{align*}
where equality here means equality in distribution. The validity of this statement follows from the symmetry of the Normal distribution, since both $X_{i} W$ and $|X_{i}| W$ have the same distribution, namely $\text{Normal}(0,X_{i}^{2})$.
\end{rmk}

\subsection{Weibull noise}\label{sec:bydistro_weibull}

$W \sim \text{Weibull}(k,\sigma)$ with shape $k > 0$ and scale $\sigma > 0$ is defined by
\begin{align*}
\prr\left\{ W \leq \alpha \right\} = 1 - \exp\left(-\left(\frac{\alpha}{\sigma}\right)^{k}\right), \qquad \alpha \geq 0.
\end{align*}
The corresponding density function is
\begin{align*}
p(u) = \frac{k}{\sigma}\left(\frac{u}{\sigma}\right)^{k-1}\exp\left(-\left(\frac{u}{\sigma}\right)^{k}\right).
\end{align*}
%
The relative entropy can be computed in a straightforward manner, as is shown in a technical note by \citet{bauckhage2013a}. The general form for the relative entropy between $W_{1} \sim \text{Weibull}(k_{1},\sigma_{1})$ and $W_{2} \sim \text{Weibull}(k_{2},\sigma_{2})$ is
\begin{align}\label{eqn:weibull_KL}
\KL(W_{1};W_{2}) = \log\frac{k_{1}}{\sigma_{1}^{k_{1}}} - \log\frac{k_{2}}{\sigma_{2}^{k_{2}}} + \left(k_{1}-k_{2}\right)\left(\log\sigma_{1} - \frac{\gamma}{k_{1}}\right) + \left(\frac{\sigma_{1}}{\sigma_{2}}\right)^{k_{2}}\Gamma\left(\frac{k_{2}}{k_{1}}+1\right) - 1,
\end{align}
where $\gamma = 0.56621566\ldots$ is the Euler-Mascheroni constant.

\begin{prop}[Estimator computation]\label{prop:weibull_comp}
Let $W \sim \text{Weibull}(k,\sigma)$. Then following our notation in (\ref{eqn:comp_Mk})--(\ref{eqn:comp_Mk_signs}) and (\ref{eqn:comp_breakdown}), we have
\begin{align*}
M_{0,1}^{0}(u) & =
\begin{cases}
\displaystyle 0 & u \leq 0 \medskip\\
\displaystyle 1 - \exp\left( -\left(u/\sigma\right)^{k} \right) & u > 0
\end{cases}\\
M_{0,1}^{l}(u) & =
\begin{cases}
\displaystyle 0 & u \leq 0 \medskip\\
\displaystyle \frac{l\sigma^{l}}{k} \Gamma\left(l/k;\left(u/\sigma\right)^{k}\right) - u^{l} \exp\left(-\left(u/\sigma\right)^{k}\right) & u > 0
\end{cases}
\end{align*}
for $l=1,2,3$, where $\Gamma(u;v)=\int_{0}^{v} e^{-t} t^{u-1} \, dt$ is the unnormalized incomplete Gamma function.\footnote{Popular numerical computation libraries almost always include efficient implementations of the incomplete gamma function. For example, \texttt{gsl\_sf\_gamma\_inc} in the GNU Scientific Library, and \texttt{special.gammainc} in SciPy. See \citet{abramowitz1964a} for more background.}
\end{prop}

Thanks to Proposition \ref{prop:weibull_comp}, we know that we can compute $\xm$ and $\xa$ under Weibull noise. Now we look at the statistical guarantees that are available for such an estimation procedure.

\begin{prop}[Deviation bounds]\label{prop:weibull_devbd}
Additive case: consider noise $\ndist = \text{Weibull}(2,\sigma)-\sigma\sqrt{\pi}/2$ with prior $\prior = \ndist$, setting $s^{2} = n(\exx_{\ddist}X^{2}+\sigma^{2}(1-\pi/4)) / 2\log(\delta^{-1})$. Then we have
\begin{align*}
|\xa - \exx_{\ddist}X| \leq \sqrt{\frac{2(\exx_{\ddist}X^{2}+\sigma^{2}(1-\pi/4))\log(\delta^{-1})}{n}}
\end{align*}
with probability no less than $1-2\delta$ over the draw of the sample.

Multiplicative case: consider noise $\ndist = \text{Weibull}(k,\sigma(k))$ with prior $\prior = \text{Weibull}(k,1)$, where $\sigma(k) = (\Gamma(1+1/k))^{-1}$. To keep the notation clean, we write
\begin{align*}
c_{k} = \frac{1}{\Gamma^{k}(1+1/k)} + k\log\Gamma(1+1/k) - 1
\end{align*}
and have that setting
\begin{align*}
s^{2} = \frac{n \Gamma(1+2/k)\exx_{\ddist}X^{2}}{2\Gamma^{2}(1+1/k)(c_{k}+\log(\delta^{-1}))}
\end{align*}
it follows that
\begin{align*}
|\xm - \exx_{\ddist}X| \leq \sqrt{\frac{2\Gamma(1+2/k)\exx_{\ddist}X^{2}(c_{k}+\log(\delta^{-1}))}{\Gamma^{2}(1+1/k)n}}
\end{align*}
with probability no less than $1-2\delta$ over the draw of the sample.
\end{prop}

Using Propositions \ref{prop:weibull_comp}--\ref{prop:weibull_devbd} and equation (\ref{eqn:comp_breakdown}), computation is done using the general form
\begin{align*}
\xhat = \frac{s}{n} \sum_{i=1}^{n} F_{W}(a_{i},b_{i})
\end{align*}
which specializes to
\begin{align*}
\xm & \text{ case: } \enspace W \sim \text{Weibull}(k,\sigma), \enspace \sigma = \frac{1}{\Gamma(1+1/k)}, \enspace a_{i}=0, \enspace b_{i} = \frac{X_{i}}{s}\\
\xa & \text{ case: } \enspace W \sim \text{Weibull}(2,\sigma), \enspace a_{i}=\frac{X_{i}-\sigma\sqrt{\pi}/2}{s}, \enspace b_{i} = \frac{1}{s}.
\end{align*}
Note that unlike our implementation in the previous section \ref{sec:bydistro_normal}, here $b_{i}$ can be both positive and negative. The need for this arises naturally as the Weibull distribution is asymmetric, and thus conditioned on $X_{i}$, we cannot in general say that $X_{i}\,W$ and $|X_{i}|W$ have the same distribution. As such, both cases in (\ref{eqn:comp_Mk_signs}) will be utilized for computations in the Weibull case.

\begin{rmk}[Special case]
The Weibull distribution includes many other well-known distributions as special cases. In particular, setting shape $k=1$ and scale $\sigma = 1/r$ for $r>0$ yields an Exponential distribution with rate parameter $r>0$.
\end{rmk}

\subsection{Student-t noise}\label{sec:bydistro_student}

We say a random variable $W \sim \text{Student}(\dof)$ has the ``Student-t'' distribution with $\dof>0$ degrees of freedom when it has the following probability density function:
\begin{align}\label{eqn:student_density}
p(u) = \frac{\Gamma((\dof+1)/2)}{\sqrt{\dof\pi}\Gamma(\dof/2)}\left(1 + \frac{u^{2}}{\dof}\right)^{-(\dof+1)/2}, \qquad u \in \RR.
\end{align}
where the integer $\dof$ is called the ``degrees of freedom'' parameter of the distribution, and $\Gamma(u)=\int_{0}^{\infty} e^{-t} t^{u-1} \, dt$ is the usual gamma function of Euler \citep{abramowitz1964a}. Throughout this section, for $q>0$ we shall write $A_{\dof} \defeq \Gamma((\dof+1)/2) / \sqrt{\dof\pi}\Gamma(\dof/2)$ for the normalization constant.

The relative entropy computation for the Student case cannot be given in closed form; if we want to compute it directly for two Student distributions with degrees of freedom $\dof_{1}$ and $\dof_{2}$, we have
\begin{align*}
\KL(\dof_{1};\dof_{2}) & = \log\frac{A_{\dof_{1}}}{A_{\dof_{2}}} - \frac{\dof_{1}+1}{2} \exx_{\dof_{1}} \log\left(1 + \frac{t^{2}}{\dof_{1}}\right) + \frac{\dof_{2}+1}{2} \exx_{\dof_{1}} \log\left(1 + \frac{t^{2}}{\dof_{2}}\right)\\
& = \log\frac{A_{\dof_{1}}}{A_{\dof_{2}}} - \frac{\dof_{1}+1}{2}\left(\Psi\left(\frac{\dof_{1}+1}{2}\right) - \Psi\left(\frac{\dof_{1}}{2}\right)\right) + \exx_{\dof_{1}} \log\left(1 + \frac{t^{2}}{\dof_{2}}\right)
\end{align*}
noting that the final summand requires numerical integration to evaluate. We write $\Psi(u) \defeq d\log\Gamma(u)/du$ to denote the derivative of the log-Gamma function, often called the digamma function.\footnote{There are many references for computing the Student relative entropy, e.g.~\citet{villa2018a} and the papers cited within for a recent reference. Numerical integration is required for the last term. The digamma function is tractable, see \citet{abramowitz1964a} for more details. Example implementations are the \texttt{gsl\_sf\_psi} function in the GNU Scientific Library, and \texttt{special.digamma} in SciPy.}

%
\begin{prop}[Estimator computation]\label{prop:student_comp}
Let $W \sim \text{Student}(\dof)$ for $\dof > 5$. Then following our notation in (\ref{eqn:comp_Mk}) and (\ref{eqn:comp_breakdown}), we have
\begin{align*}
M_{0,1}^{0}(u) & = \int_{-\infty}^{u} p(t) \, dt\\
M_{0,1}^{1}(u) & = (-1) \frac{\dof}{\dof-1} A_{\dof} \left(1 + \frac{u^{2}}{\dof}\right)^{-(\dof-1)/2}\\
M_{0,1}^{2}(u) & = \frac{\dof^{3/2}}{\sqrt{\dof-2}(\dof-1)} \, \frac{A_{\dof}}{A_{\dof-2}} N_{\dof-2}^{0}\left(u\sqrt{\frac{\dof-2}{\dof}}\right) - \frac{\dof}{(\dof-1)} A_{\dof} \, u \left(1 + \frac{u^{2}}{\dof}\right)^{-(\dof-1)/2}\\
M_{0,1}^{3}(u) & = \frac{2\dof^{2}}{(\dof-2)(\dof-1)} \frac{A_{\dof}}{A_{\dof-2}} N_{\dof-2}^{1}\left(u \sqrt{\frac{\dof-2}{\dof}}\right) - \frac{\dof}{(\dof-1)} A_{\dof} \, u^{2} \left(1 + \frac{u^{2}}{\dof}\right)^{-(\dof-1)/2}
\end{align*}
where $N_{m}^{k}(\cdot)$ denotes $M_{0,1}^{k}(\cdot)$ under a Student-t distribution with $m \leq \dof$ degrees of freedom.
\end{prop}

Computation of $\xm$ and $\xa$ under Student-t noise is slightly more complicated than in the previous cases seen above, but as shown in Proposition \ref{prop:student_comp}, it is tractable. Next we look at guarantees on the statistical accuracy.

\begin{prop}[Deviation bounds]\label{prop:student_devbd}
In both the additive and multiplicative cases below, the relative entropy can be bounded above by
\begin{align*}
\KL(\ndist;\prior) \leq c(\alpha,\dof) \defeq \frac{(\dof+1)}{2}\left( \log\left(1+\frac{\alpha^{2}}{\dof}\right) + \frac{4|\alpha|\Gamma((\dof+1)/2)}{\sqrt{\dof\pi}\,\Gamma(\dof/2)(\dof-1)} \right)
\end{align*}
where $\alpha \in \RR$ is a free parameter specified below, and $\dof$ is the degrees of freedom parameter of the underlying Student-t distribution.

Additive case: consider noise $\ndist = \text{Student}(\dof)$ with prior $\prior = \text{Student}(\dof) - \alpha$ for an arbitrary constant $\alpha \in \RR$. Setting the scaling parameter as
\begin{align*}
s^{2} = \frac{n(\exx_{\ddist}X^{2}+\dof/(\dof-2))}{2(c(\alpha,\dof) + \log(\delta^{-1}))}
\end{align*}
we have with probability at least $1-2\delta$ that
\begin{align*}
|\xa - \exx_{\ddist}X| \leq \sqrt{\frac{2(\exx_{\ddist}X^{2}+\dof/(\dof-2))(c(\alpha,\dof)+\log(\delta^{-1}))}{n}}.
\end{align*}

Multiplicative case: consider noise $\ndist = \text{Student}(\dof)+\alpha$ for an arbitrary constant $\alpha \in \RR$, with prior $\prior = \text{Student}(\dof)$. Setting the scaling parameter as
\begin{align*}
s^{2} = \frac{n (\alpha^{2}(\dof-2)+\dof) \exx_{\ddist}X^{2}}{2(\dof-2)(c(\alpha,\dof) + \log(\delta^{-1}))}
\end{align*}
we have with probability at least $1-2\delta$ that
\begin{align*}
|\xm-\exx_{\ddist}X| \leq \sqrt{\frac{2(\alpha^{2}(\dof-2)+\dof)\exx_{\ddist}X^{2}(c(\alpha,\dof) + \log(\delta^{-1}))}{(\dof-2)n}}
\end{align*}
\end{prop}

\noindent With Propositions \ref{prop:student_comp}--\ref{prop:student_devbd}, via (\ref{eqn:comp_breakdown}) we can compute using the general form
\begin{align*}
\xhat & = \frac{s}{n} \sum_{i=1}^{n} F_{W}(a_{i},b_{i}),
\end{align*}
and with $W \sim \text{Student}(\dof)$ for large enough $\dof$ (as specified in Proposition \ref{prop:student_comp}), we specialize as
\begin{align*}
\xm & \text{ case: } \enspace a_{i}=\frac{\alpha X_{i}}{s}, \enspace b_{i} = \frac{|X_{i}|}{s}\\
\xa & \text{ case: } \enspace a_{i}=\frac{X_{i}}{s}, \enspace b_{i} = \frac{1}{s}.
\end{align*}
Once again we remark that the $b_{i}$ computations can be done with either $X_{i}$ or $|X_{i}|$, since by symmetry the resulting noise distribution conditioned on $X_{i}$ is the same, just as in the Normal case (section \ref{sec:bydistro_normal}).

\section{Empirical analysis}\label{sec:empirical}

In this section, we use controlled simulations to investigate the behavior of the estimators $\xm$ and $\xa$ under various types of noise distributions, and compare this with other benchmark estimators. In addition to comparing the actual distribution of the deviations $|\xhat - \exx_{\ddist}X|$ with confidence intervals derived in section \ref{sec:theory}, we also pay particular attention to how performance depends on the underlying distribution, the mean to standard deviation ratio, and the sample size.\footnote{Full implementations of all the experiments carried out in this work are made available at the following online repository: \url{https://github.com/feedbackward/1dim.git}.}

\subsection{Experimental setup}

\paragraph{Data generation}
For each experimental setting and each independent trial, we generate a sample $X_{1},\ldots,X_{n}$ of size $n$, compute some estimator $\{X_{i}\}_{i=1}^{n} \mapsto \xhat$, and record the deviation $|\xhat-\exx_{\ddist}X|$. The sample sizes range over $n \in \{10,20,30,\ldots,100\}$, and the number of trials is $10^{4}$. We draw data from two distribution families: the Normal family with mean $\mu$ and variance $\sigma^{2}$, and the log-Normal family, with log-mean $\mu_{\text{log}}$ and log-variance $\sigma_{\text{log}}^{2}$, under multiple parameter settings. Regarding the variance, we have ``low,'' ``mid,'' and ``high'' settings,  which correspond to $\sigma=0.5, 5.0, 50.0$ in the Normal case, and $\sigma_{\text{log}}=1.1, 1.35, 1.75$ in the log-Normal case. Over all settings, the log-location parameter of the log-Normal data is fixed at $\mu_{\text{log}}=0$. 

\paragraph{Moment control}
Of particular interest here is the impact of different mean to standard deviation ratios: we test $r(X) = \exx X / \sd(X)$ ranging over $[-2.0, 2.0]$. The standard deviation is fixed as just described, and the mean value is determined automatically as $\exx X = r(X)\sd(X)$ for each value of $r(X)$ to be tested. Shifting the Normal data is trivially accomplished by taking the desired $\mu = r(X)\sd(X)$. Shifting the log-Normal data is accomplished by subtracting the true mean (pre-shift) equal to $\exp(\mu_{\text{log}}+\sigma_{\text{log}}^{2}/2)$ to center the data, and subsequently adding the desired location.

\paragraph{Methods being tested}
We compare canonical location estimators with numerous examples from the robust estimator class described and analyzed in sections \ref{sec:overview}--\ref{sec:theory}. The methods being compared are organized in Table \ref{table:empirical_methods}. Essentially, we are comparing the empirical mean and median with different manifestations of the estimators $\xm$ and $\xa$. As additional reference, we also compare with two sub-Gaussian estimators introduced in \ref{sec:overview}, namely the median-of-means estimator and an M-estimator constructed in the style of \citet{catoni2012a}. For the former, we partition into equal-sized subsets, with $k = \lceil \log(\delta^{-1}) \rceil$. Regarding bound computations, from Theorem 4.1 of \citet{devroye2016a}, the median-of-means estimator satisfies (\ref{eqn:subG_estimators}) with constant $c = 2\sqrt{2e}$, as long as $\delta > \exp(1-n/2)$. In the event this is not satisfied, we just return the sample mean. For the latter, as an influence function we use the Gudermannian function $\psi(u) = 2\atan(\exp(u))-\pi/2$, with scaling as $s^{2} = 2n\vaa(X) / \log(\delta^{-1})$. Using the results of \citet{catoni2012a} it is then straightforward to show that the resulting estimator satisfies
\begin{align*}
|\xhat-\exx X| \leq 2 \sqrt{\frac{2\vaa(X)\log(\delta^{-1})}{n}}
\end{align*}
with probability no less than $1-2\delta$.

Since all non-classical estimators here depend on a confidence level parameter, for all experiments we set $\delta = 0.01$. Furthermore, for precise control of the experimental conditions, we use the true variance and/or second moments of the data distribution for scaling $\xm$ and $\xa$ as specified in section \ref{sec:theory}, which is known since we are in control of the underlying data distribution. The empirical median is computed (after sorting) as the middle point when $n$ is odd, or the average of the two middle points when $n$ is even. For $\xm$ under Weibull noise, $k = 2.0$ and $\sigma$ is determined automatically. For $\xa$ under Weibull noise, $k=2.0$ and $\sigma = 1.0$. For both $\xm$ and $\xa$ under Student-t noise, shift parameter is $\alpha=1$, and degrees of freedom parameter is $\dof=5.1$.

\begin{table}[h!]
\begin{center}
\begin{tabular}{|l|l|}
\hline
\texttt{mean} & Empirical mean.\\
\texttt{med} & Empirical median.\\
\texttt{mom} & Median-of-means \citep{devroye2016a}.\\
\texttt{mest} & M-estimator \citep{catoni2012a}.\\
\texttt{mult\_b} & $\xm$ under Bernoulli noise (sec.~\ref{sec:bydistro_bernoulli}).\\
\texttt{mult\_bc} & centered version of \texttt{mult\_b} (sec.~\ref{sec:bydistro_bernoulli}).\\
\texttt{mult\_g} & $\xm$ under Normal noise (sec.~\ref{sec:bydistro_normal}).\\
\texttt{mult\_w} & $\xm$ with Weibull noise (sec.~\ref{sec:bydistro_weibull}).\\
\texttt{mult\_s} & $\xm$ with Student noise (sec.~\ref{sec:bydistro_student}).\\
\texttt{add\_g} & $\xa$ with Normal noise (sec.~\ref{sec:bydistro_normal}).\\
\texttt{add\_w} & $\xa$ with Weibull noise (sec.~\ref{sec:bydistro_weibull}).\\
\texttt{add\_s} & $\xa$ with Student-t noise (sec.~\ref{sec:bydistro_student}).\\
\hline
\end{tabular}
\end{center}
\vspace{-0.5cm}
\caption{List of location estimators being evaluated in sections \ref{sec:empirical_meanSD}--\ref{sec:empirical_overdist}.}
\label{table:empirical_methods}
\end{table}

\subsection{Impact of mean-SD ratio}\label{sec:empirical_meanSD}

\begin{figure}[t]
\centering
\includegraphics[width=1.0\textwidth]{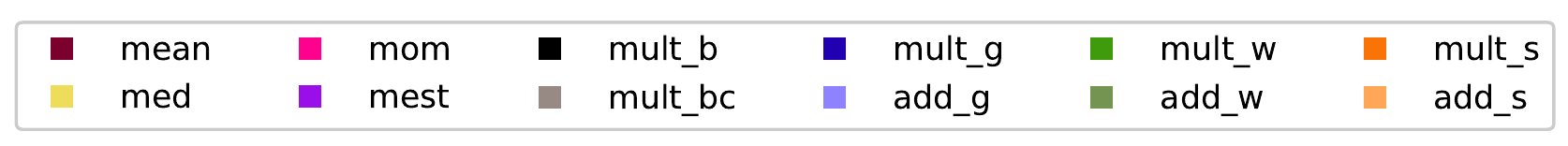}\\
\includegraphics[width=0.50\textwidth]{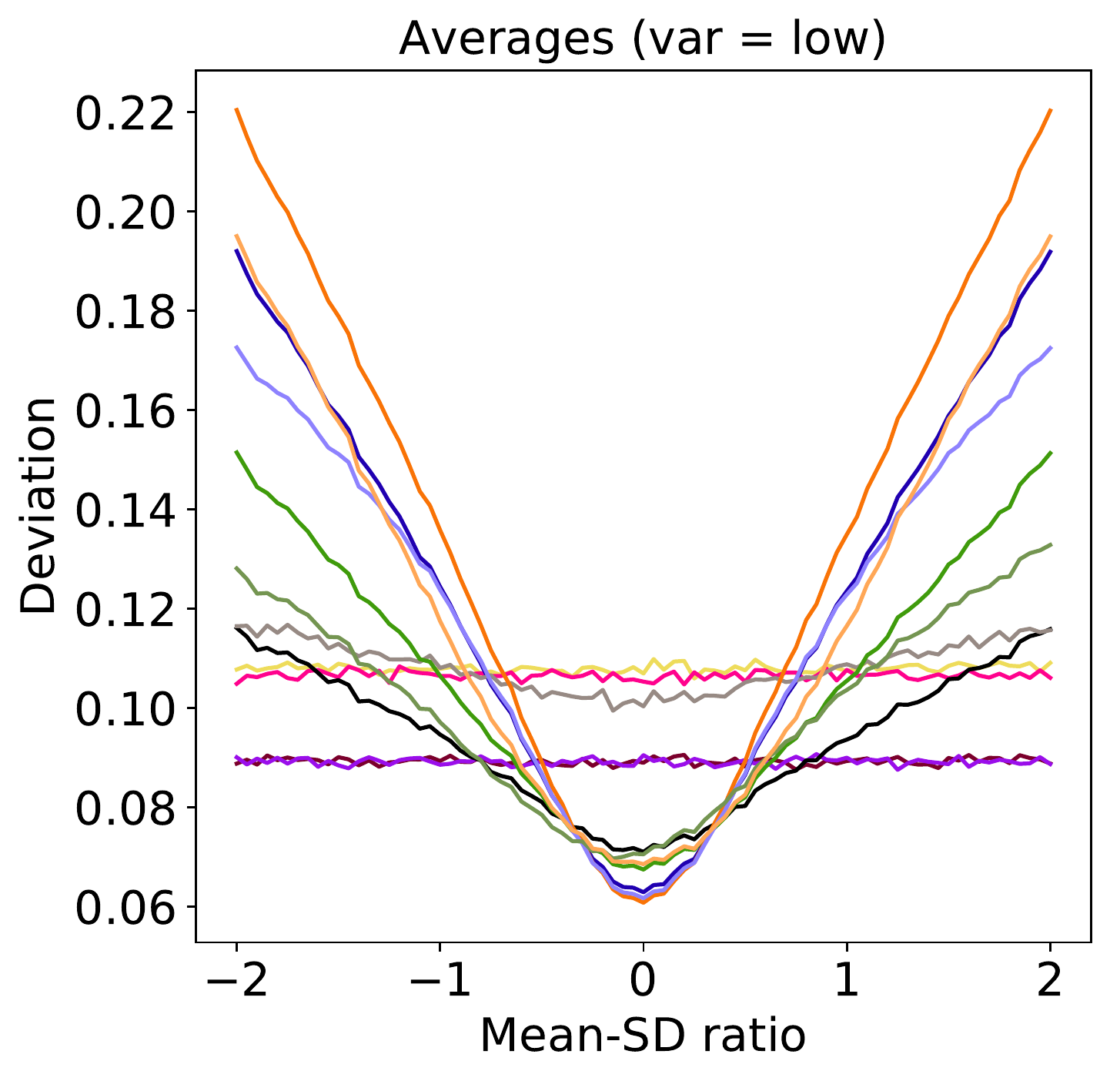}\,\includegraphics[width=0.48\textwidth]{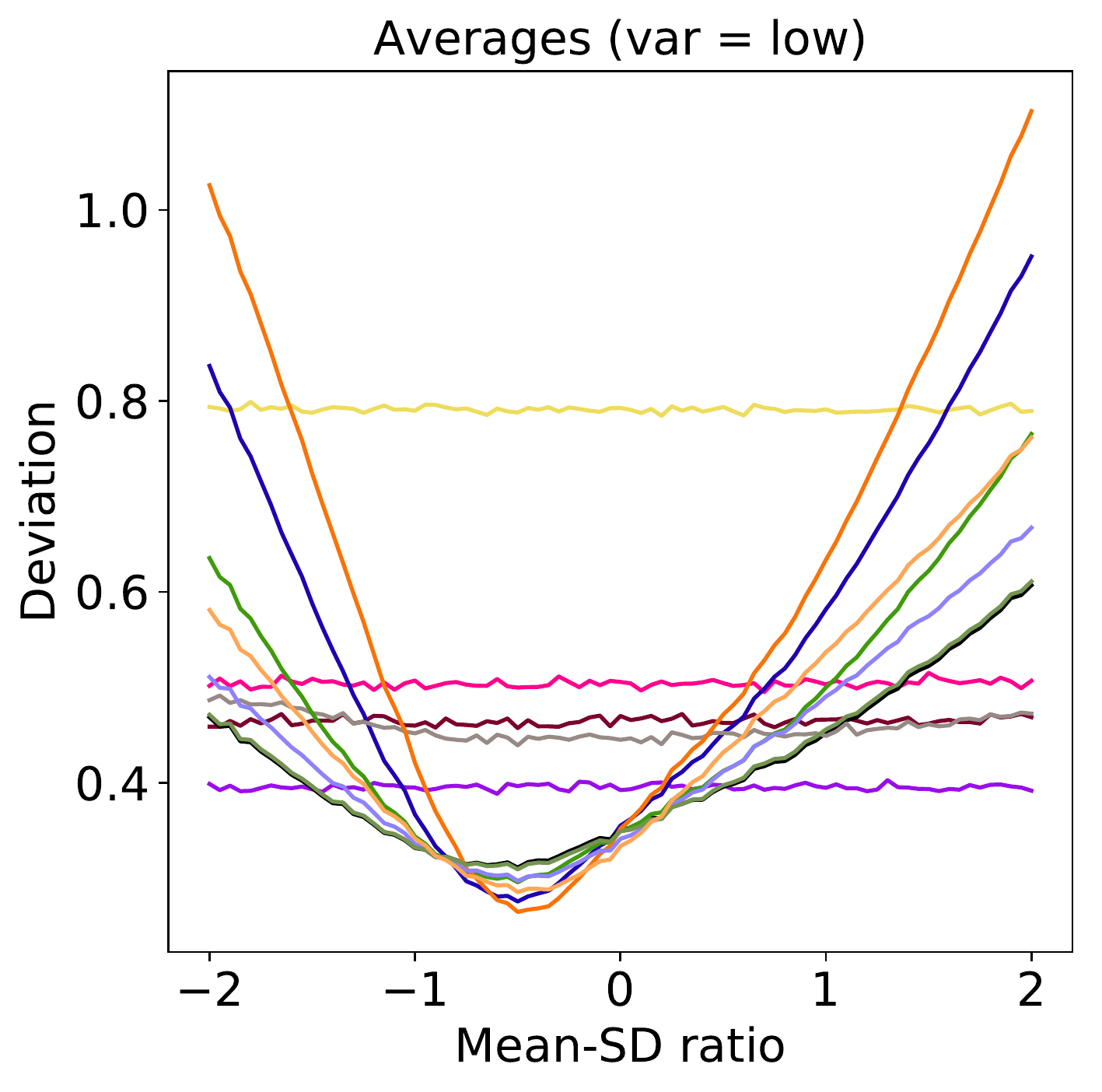}
\caption{Deviations $|\xhat - \exx_{\ddist}X|$ averaged over all trials, plotted as a function of the ratio $r(X) = \exx X / \sd(X)$. Sample size is $n=20$, variance level is low. Left: Normal data. Right: log-Normal data.}
\label{fig:overratioAves_deviations}
\end{figure}

In Figure \ref{fig:overratioAves_deviations}, we look at how the mean to standard deviation ratio impacts the performance of different estimators. Clearly, the dependence of $\xm$ and $\xa$ on the second moment, as suggested by the bounds derived in section \ref{sec:theory}, is not vacuous. A remarkably large difference in sensitivity appears between different manifestations of these new estimators, however. In order of increasing sensitivity: Bernoulli, Weibull, Gaussian, Student. Estimators with a higher sensitivity to the absolute value of the mean are more strongly biased toward a particular location value. We also note that in observing analogous plots as the variance level increases from low $\to$ mid $\to$ high, all of the $\xa$ instances (the \texttt{add\_*} methods) converge to behave identically to the Bernoulli type estimator \texttt{mult\_b}. The reason for this is simple. As $\vaa(X)$ grows, so does the value of $s$ used in Propositions \ref{prop:bernoulli_all}, \ref{prop:normal_devbd}, \ref{prop:weibull_devbd}, and \ref{prop:student_devbd}. Since we are not modifying the noise distribution in a data-driven fashion, this means that the variance of scaled noise $\epsilon / s$ decreases, taking each summand in $\xa$ of the form $\exx_{\ndist}\trunc((X/s)+(\epsilon/s))$ progressively closer to the Bernoulli case of $\exx_{\ndist}\trunc(X/s)$. This is in stark contrast to the multiplicative case $\xm$, in which the noise and the data are coupled.

Overall, the Bernoulli case is least sensitive among the uncentered estimators, and considering the ease of computation, is clearly a strong choice. Furthermore, the centered version \texttt{mult\_bc} is distinctly much less sensitive to the mean-SD ratio, which is precisely what we would expect given the discussion in section \ref{sec:bydistro_bernoulli}. In the log-Normal case, since the dependence on the parameters being controlled is asymmetric, the nature of the underlying distribution differs for positive and negative values of $r(X)$, leading to the asymmetric curves observed.

\subsection{Impact of sample size}\label{sec:empirical_overN}

\begin{figure}[t]
\centering
\includegraphics[width=0.50\textwidth]{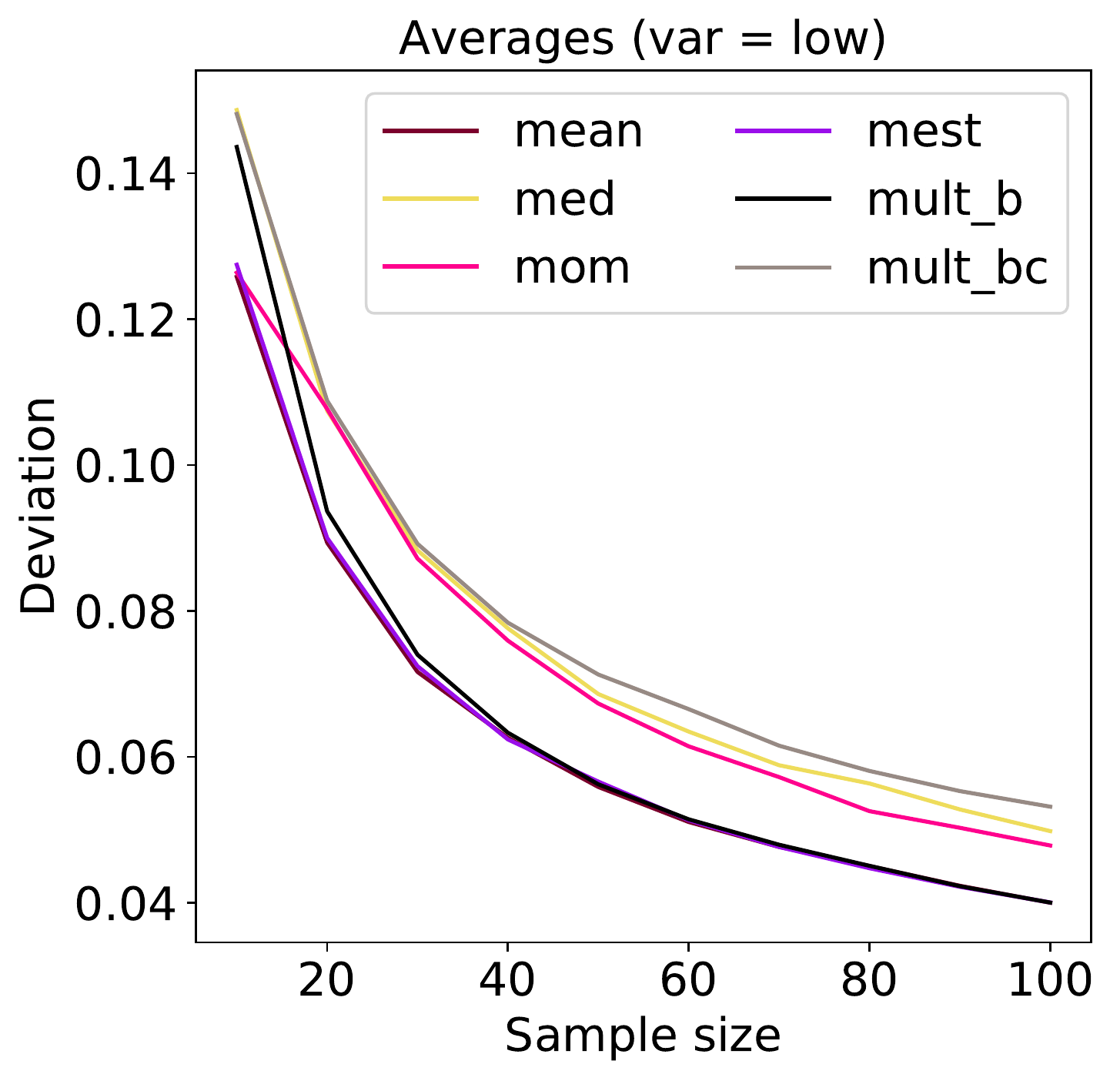}\,\includegraphics[width=0.48\textwidth]{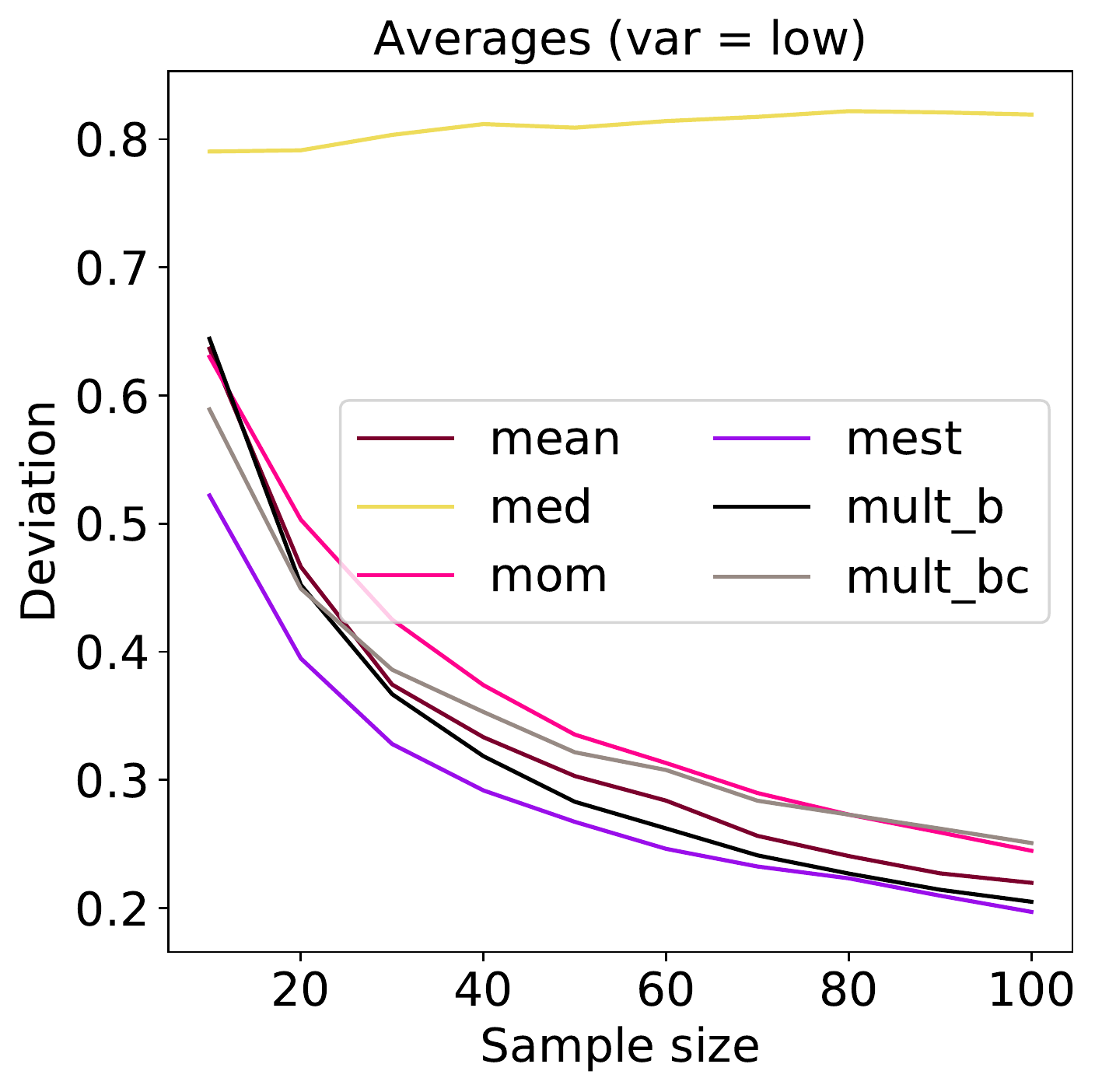}
\caption{Deviations $|\xhat - \exx_{\ddist}X|$ averaged over all trials, plotted as a function of the sample size $n$. Mean to standard deviation ratio is $r(X) = 1.0$, variance level is low. Left: Normal data. Right: log-Normal data.}
\label{fig:overNAves_deviations}
\end{figure}

In Figure \ref{fig:overNAves_deviations}, we examine how performance improves as the sample size $n$ increases; at the same time, we may observe how performance deteriorates when data is scarce. For readability, here we only compare the two classical estimators and two known sub-Gaussian estimators with the Bernoulli-type estimators from our class of interest. In addition to the consistency of $\xm$ (and its centered version) that the deviation bounds of section \ref{sec:theory} imply, we are able to confirm that $\xm$ is a very close competitor to the variance-dependent M-estimator of \citet{catoni2012a}. Without any iterative sub-routine, $\xm$ does well under the outlier-prone log-Normal case, and yet has small enough bias so as to effectively mimic the sample mean in the Normal case. Performance is well above the median-of-means estimator, but both $\xm$ and the M-estimator make use of moment information in these tests, so definitive statements about relative superiority remain difficult to make.

\subsection{Distribution of deviations}\label{sec:empirical_overdist}

\begin{figure}[t]
\centering
\includegraphics[width=0.24\textwidth]{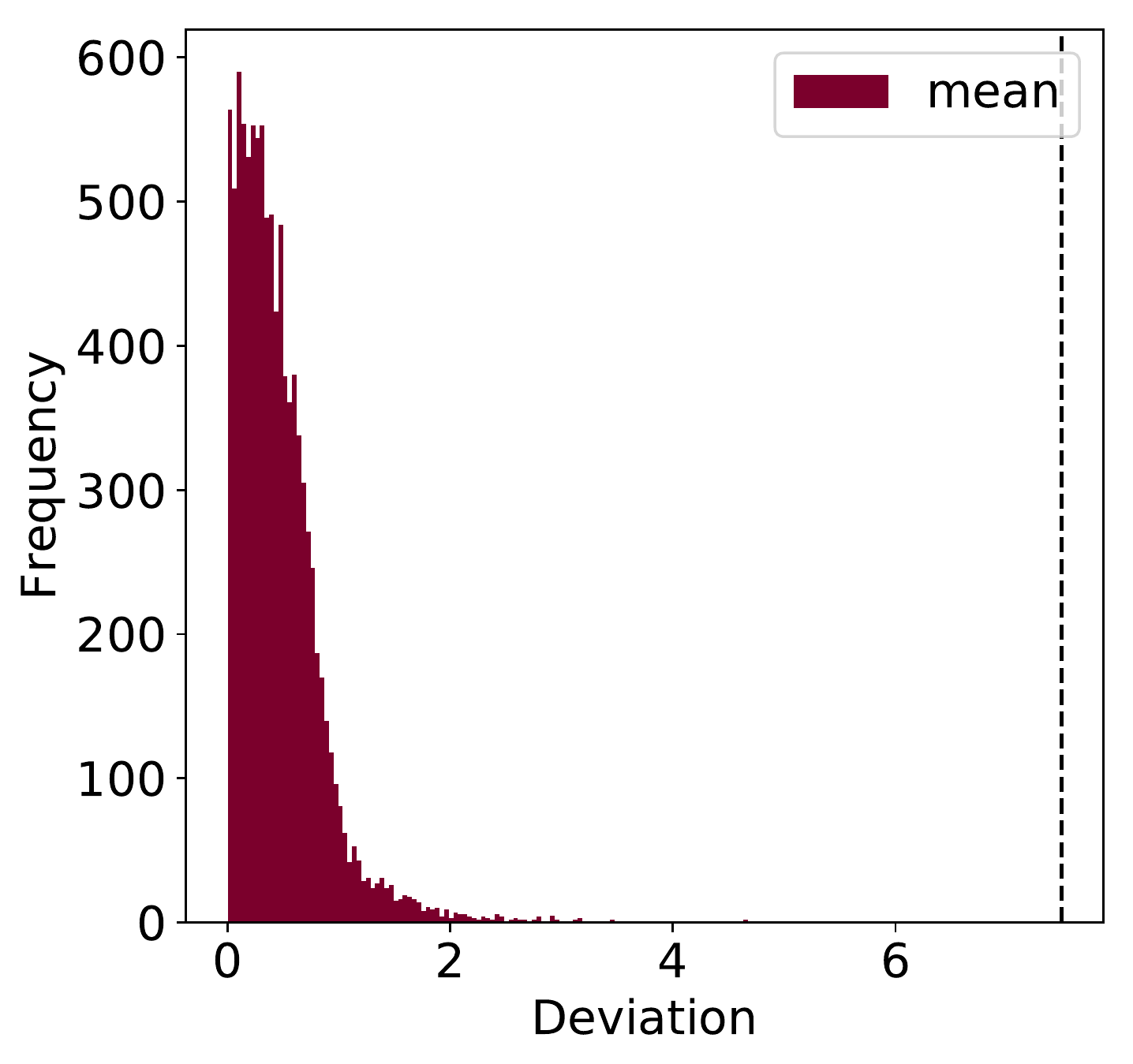}\,\includegraphics[width=0.24\textwidth]{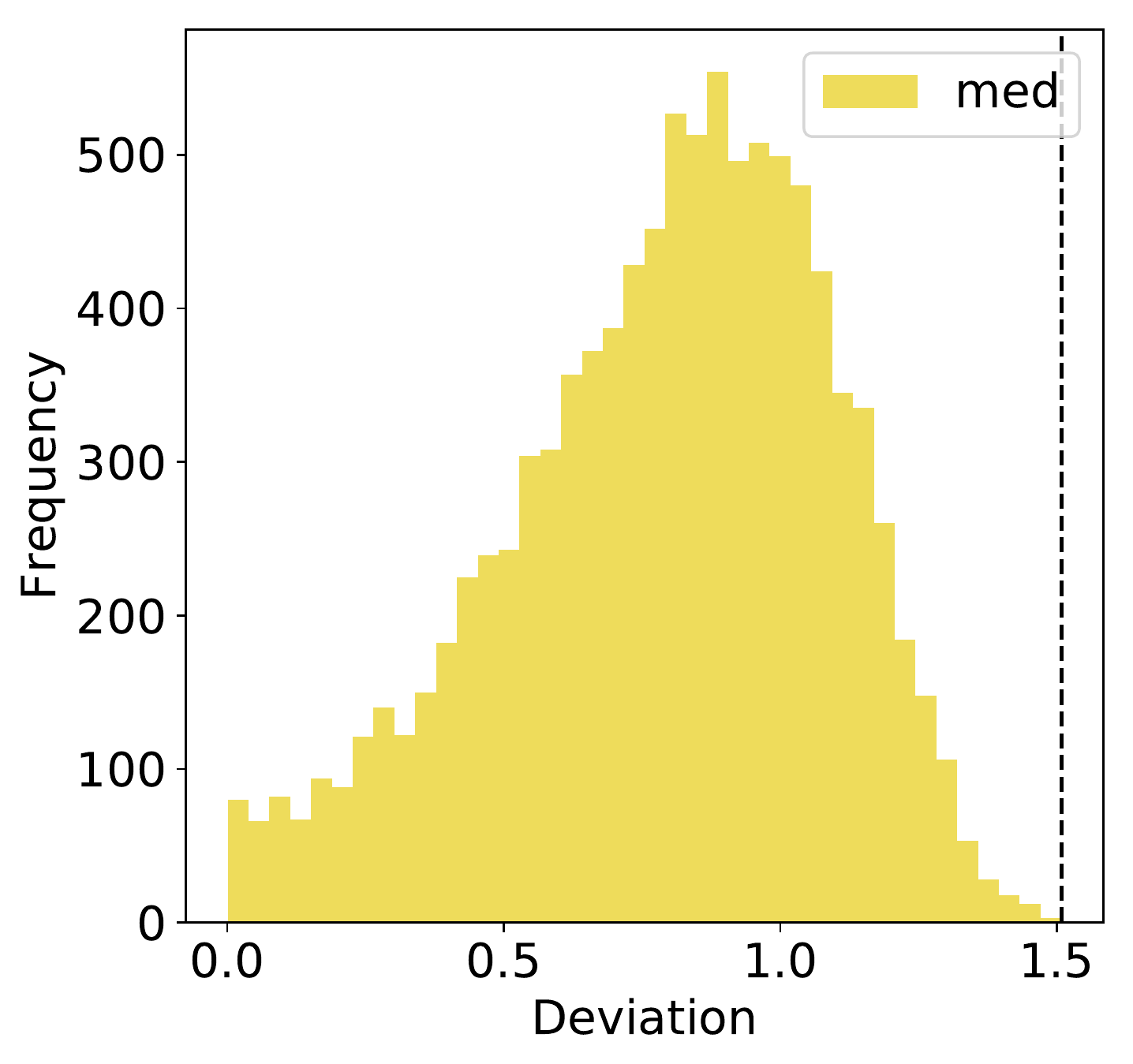}\,\includegraphics[width=0.24\textwidth]{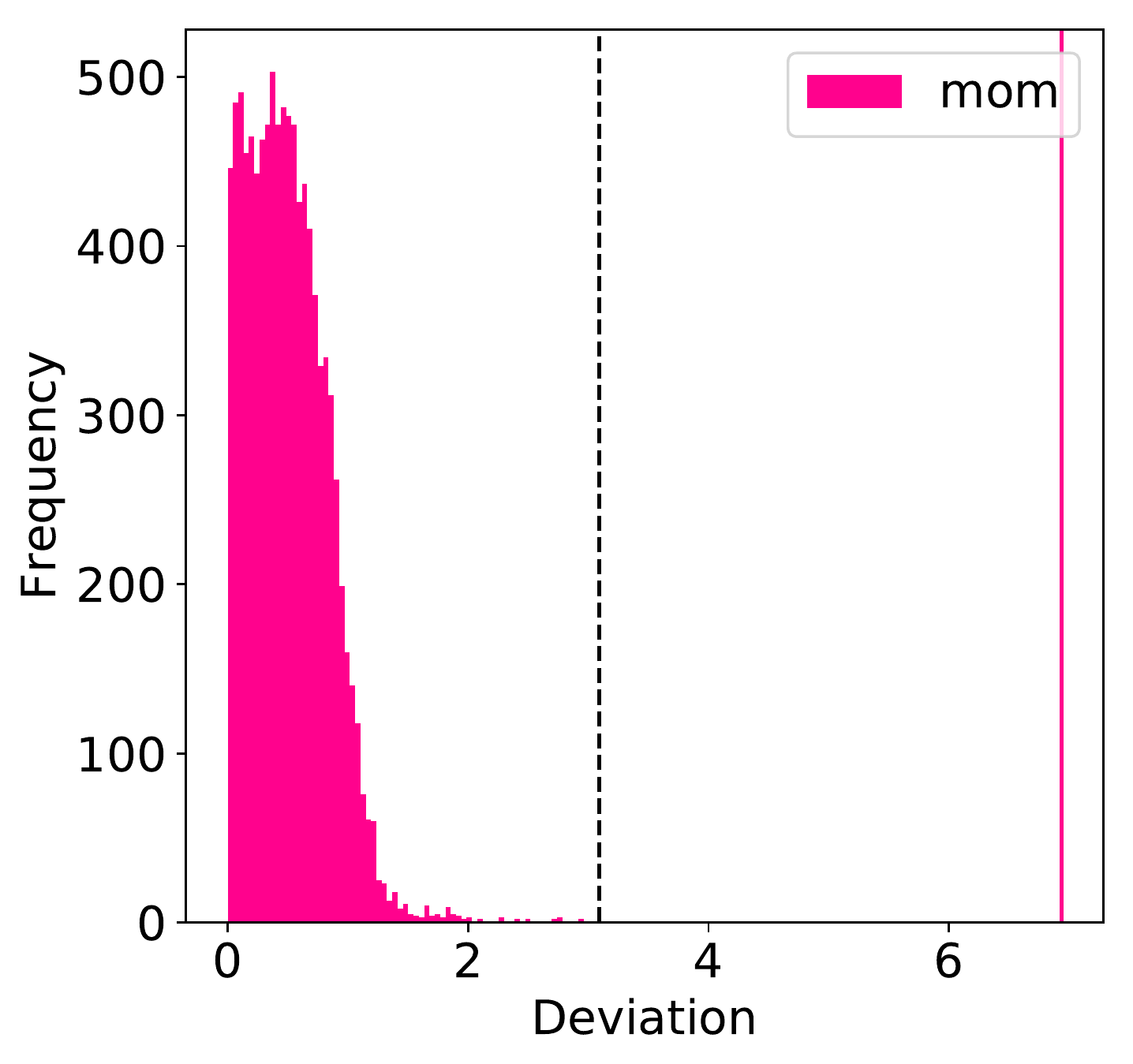}\,\includegraphics[width=0.24\textwidth]{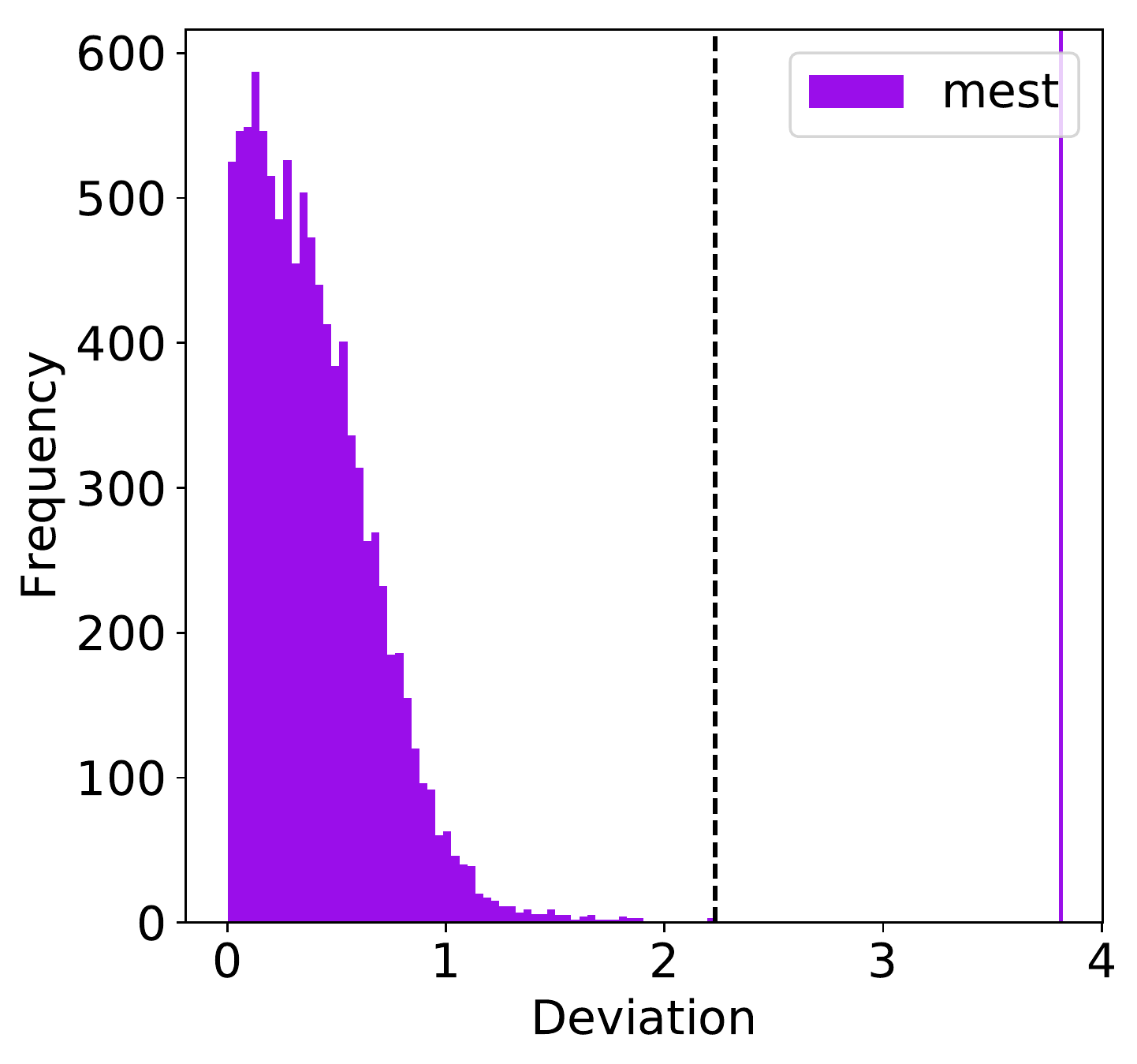}\\
\includegraphics[width=0.24\textwidth]{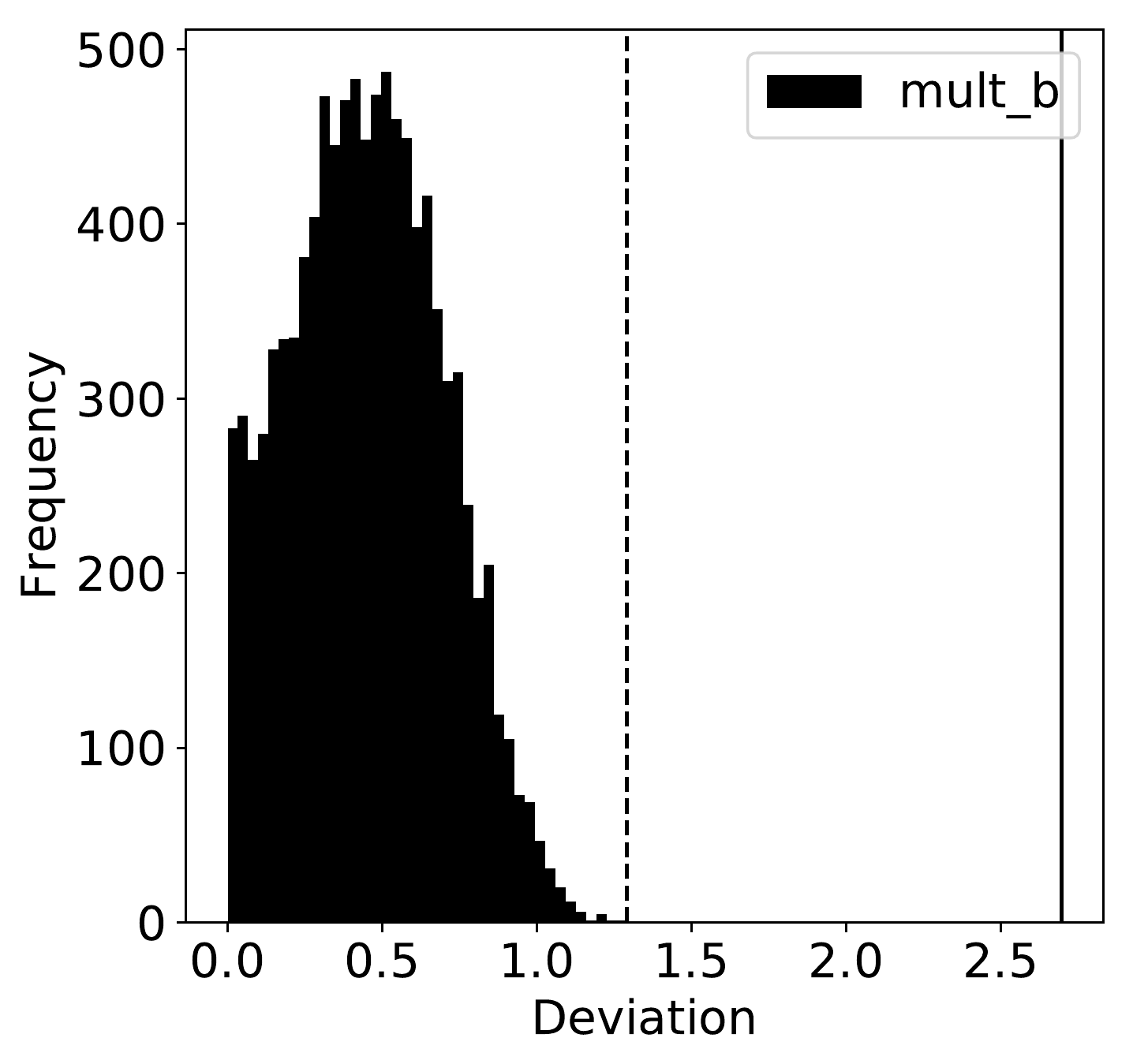}\,\includegraphics[width=0.24\textwidth]{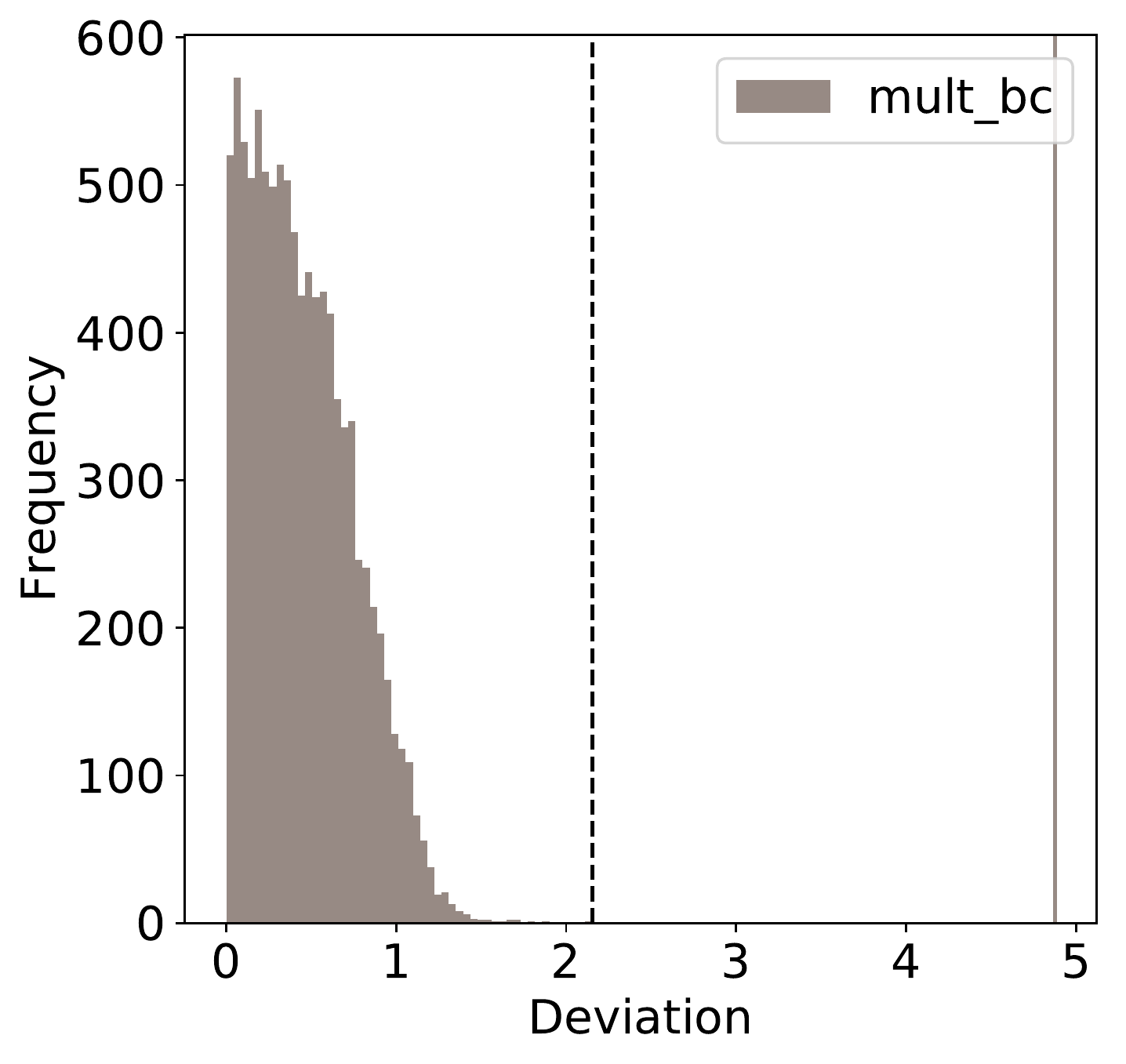}\,\includegraphics[width=0.24\textwidth]{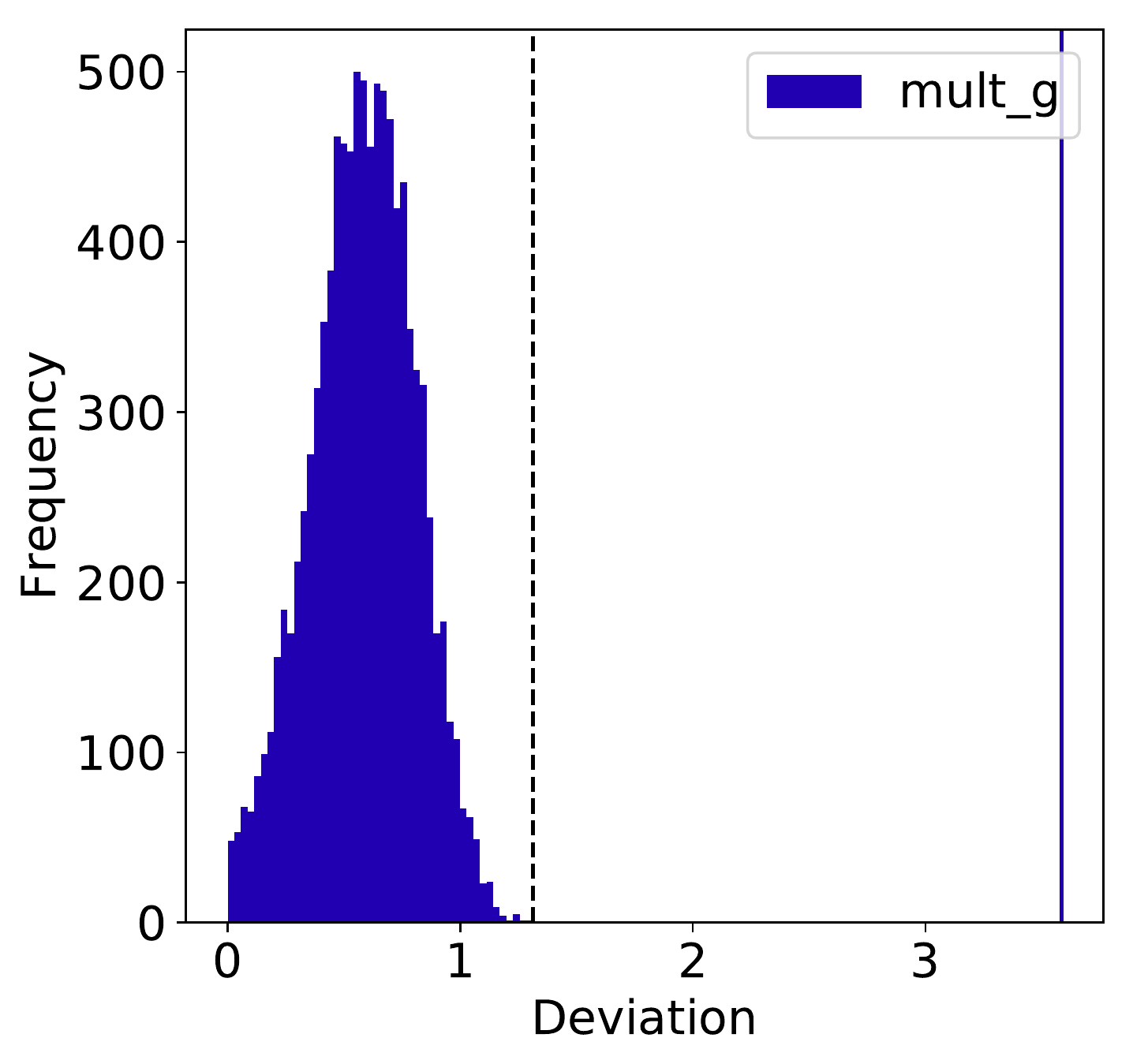}\,\includegraphics[width=0.24\textwidth]{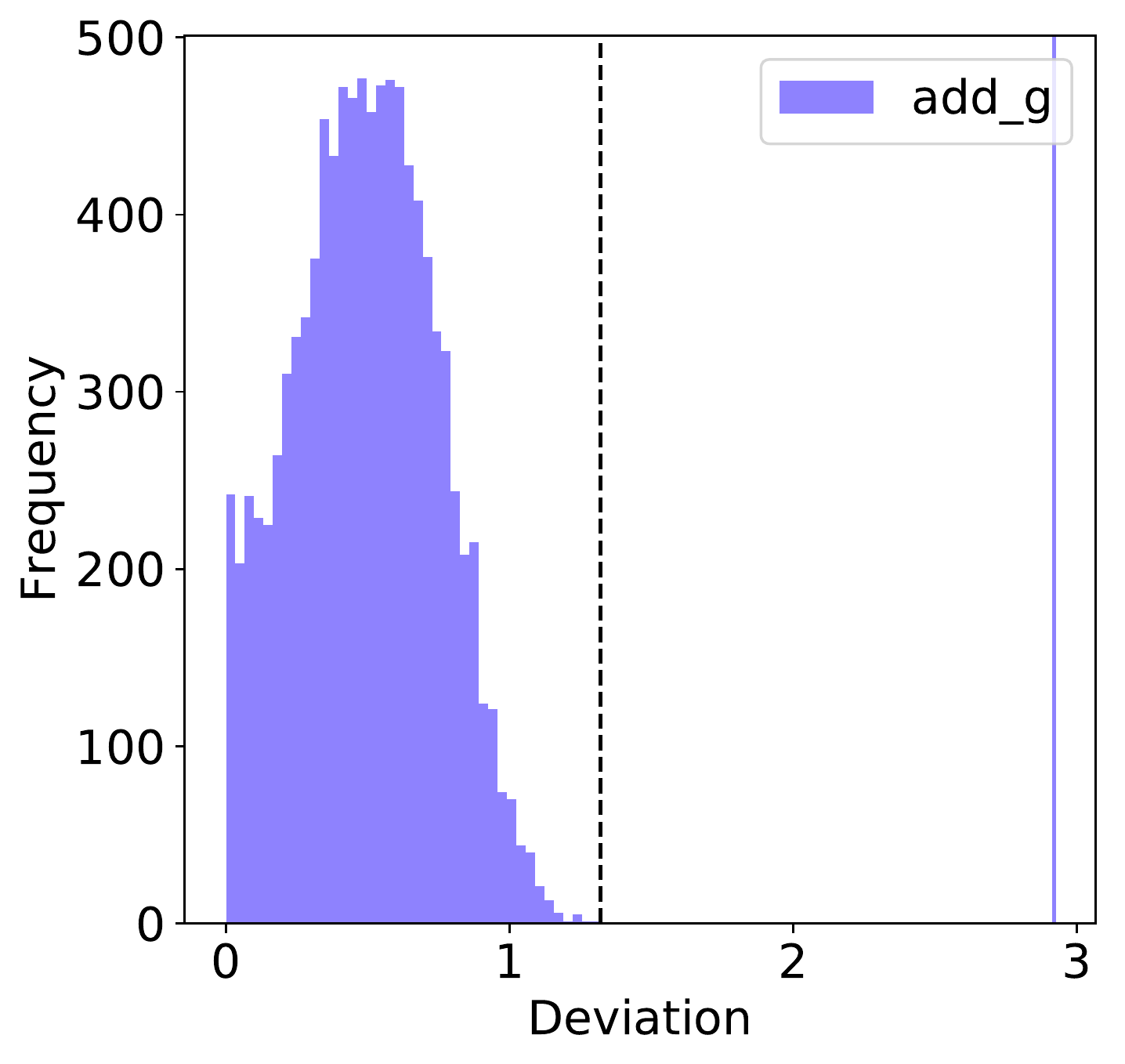}\\
\includegraphics[width=0.24\textwidth]{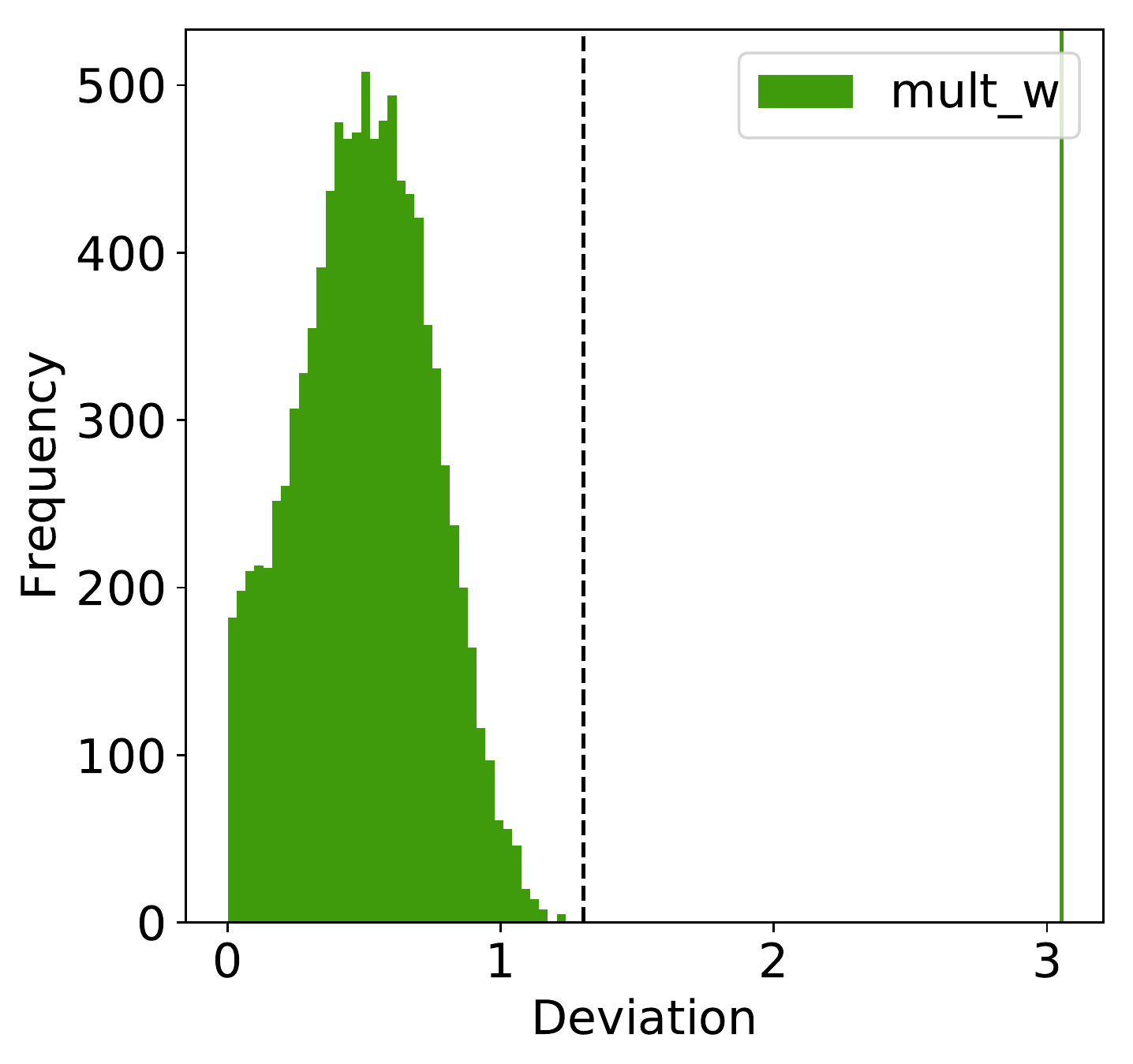}\,\includegraphics[width=0.24\textwidth]{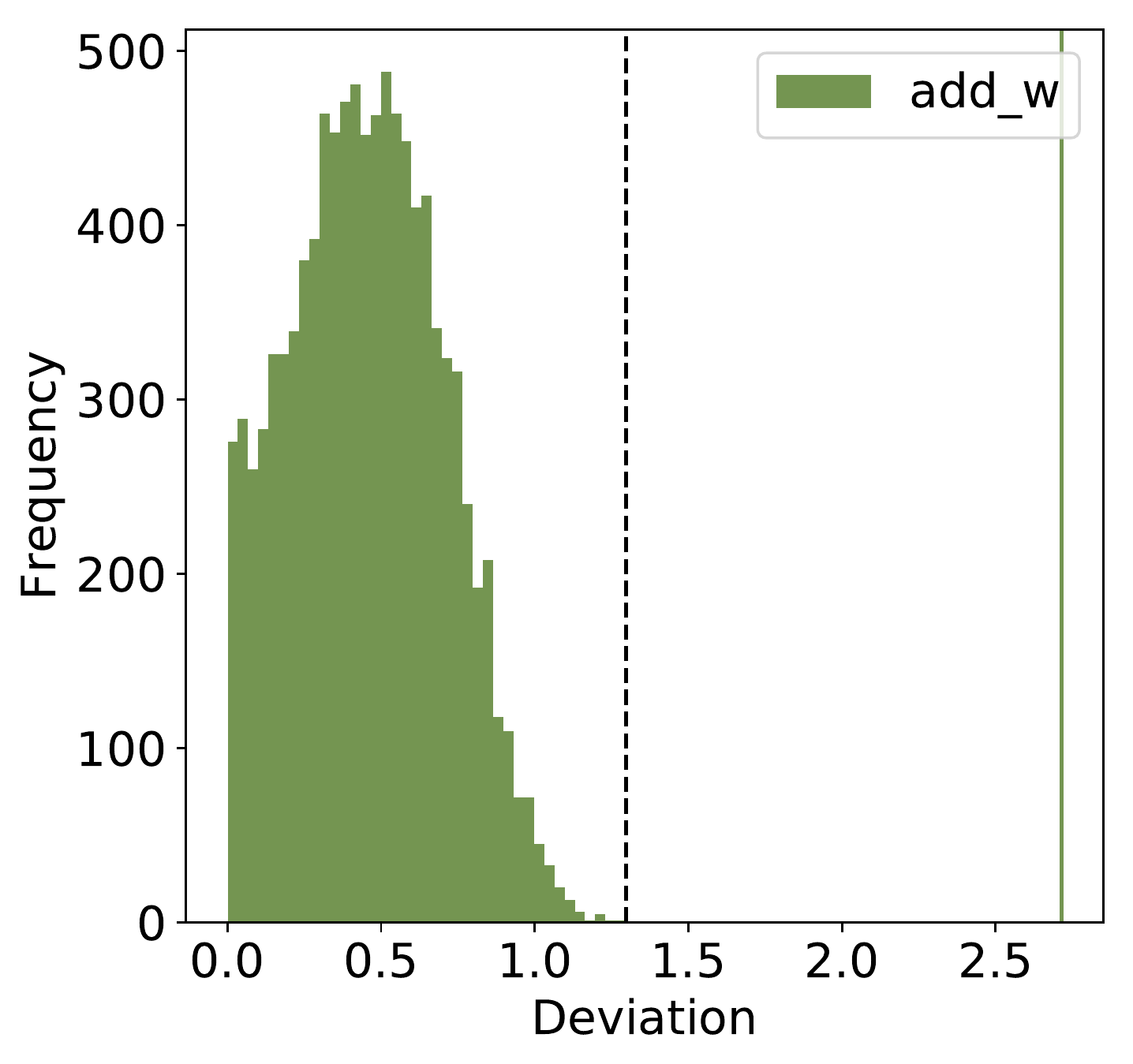}\,\includegraphics[width=0.24\textwidth]{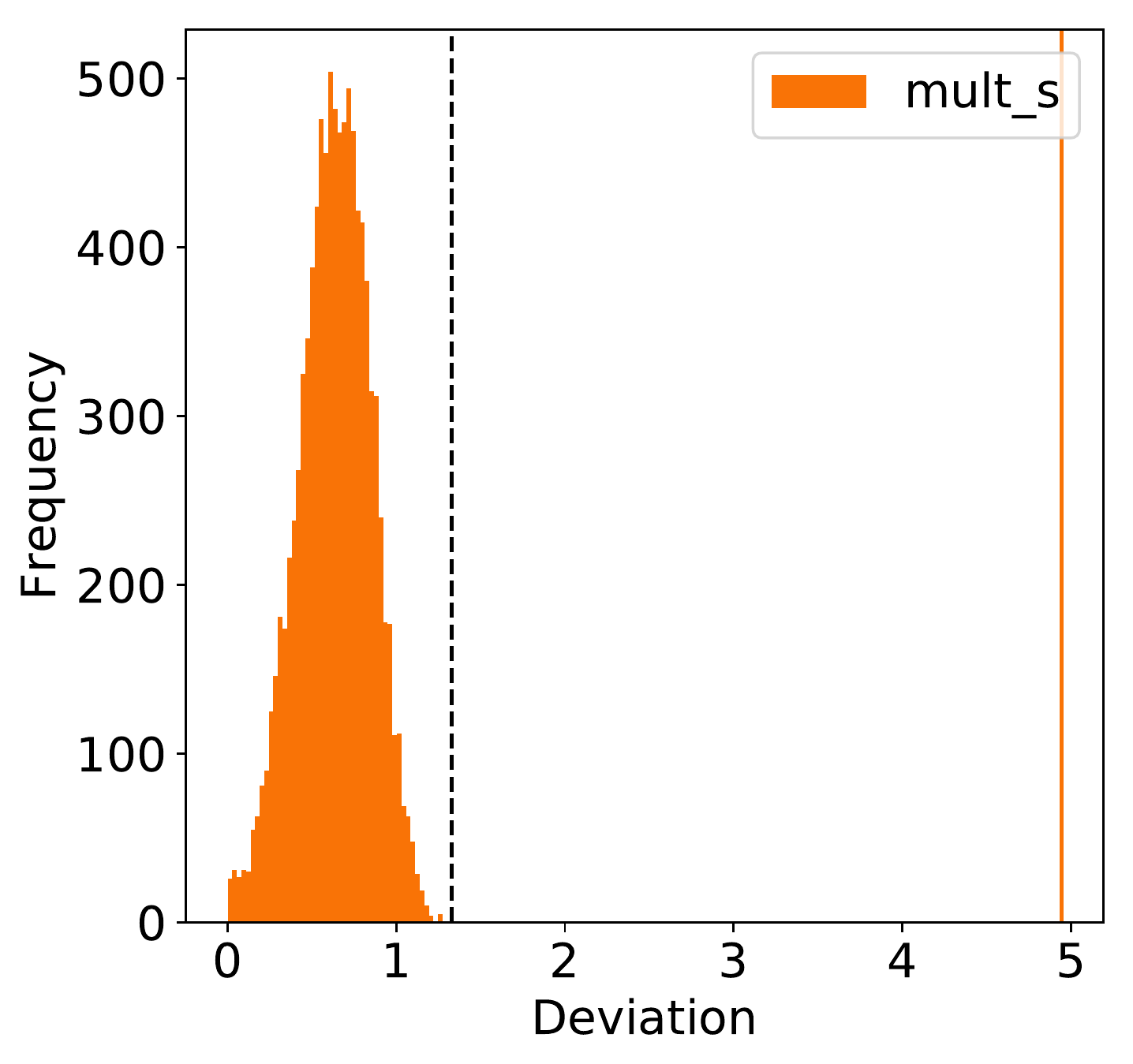}\,\includegraphics[width=0.24\textwidth]{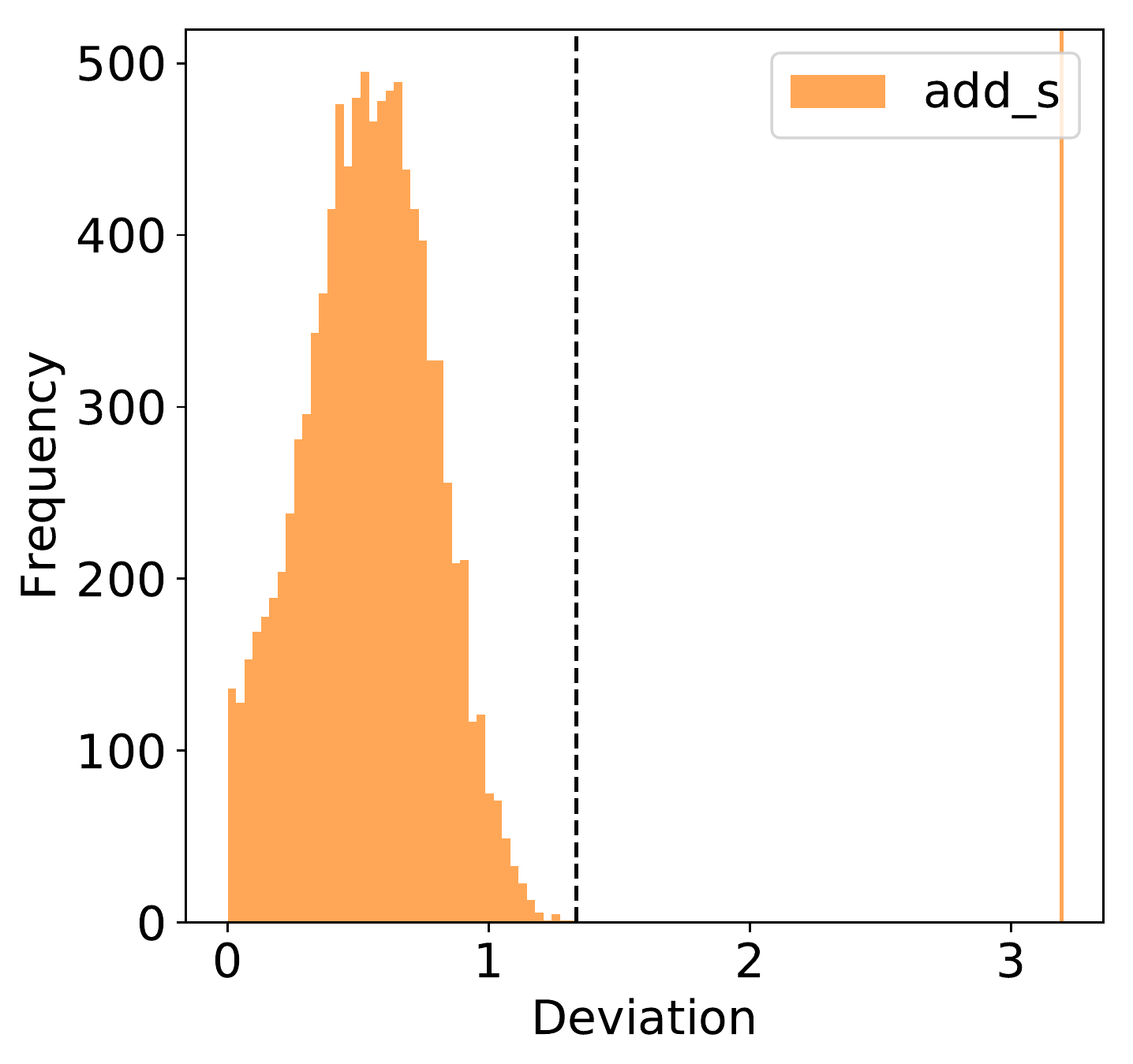}
\caption{Histograms of deviations $|\xhat - \exx_{\ddist}X|$ for all estimators being evaluated, with accompanying two-sided $1-2\delta$ confidence intervals. Data is log-Normal, sample size is $n=20$, variance level is low, mean to standard deviation ratio is $r(X)=1.0$. Since rare values are difficult to see, the dashed black vertical line indicates the largest observed deviation.}
\label{fig:hist_deviations}
\end{figure}

\begin{figure}[t]
\centering
\includegraphics[width=0.49\textwidth]{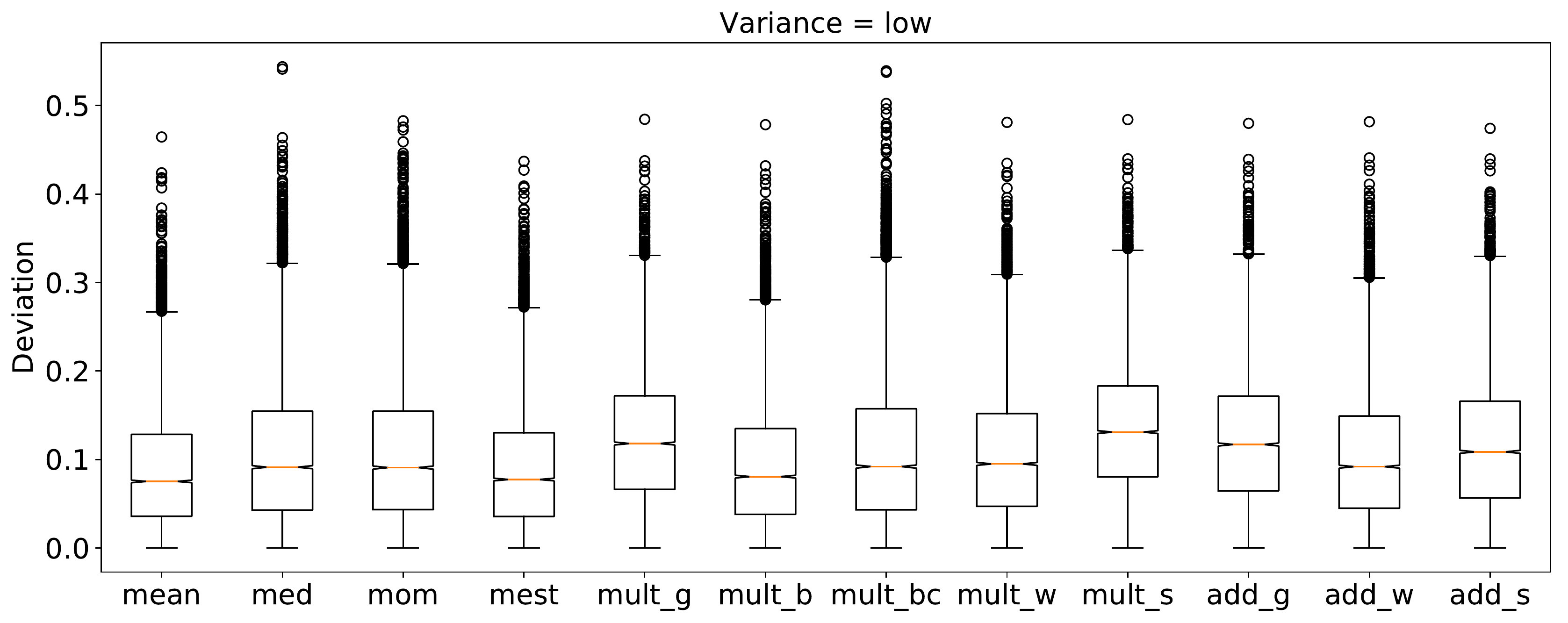}\includegraphics[width=0.49\textwidth]{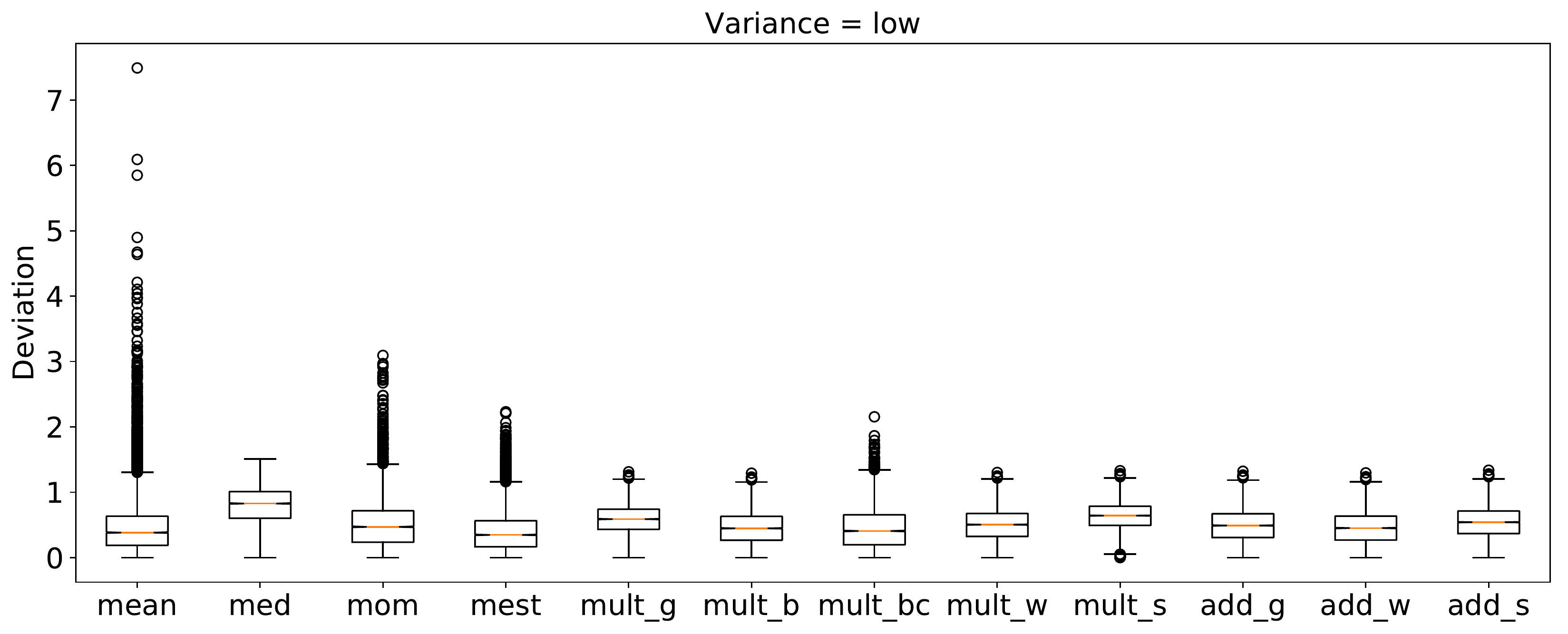}\\
\includegraphics[width=0.49\textwidth]{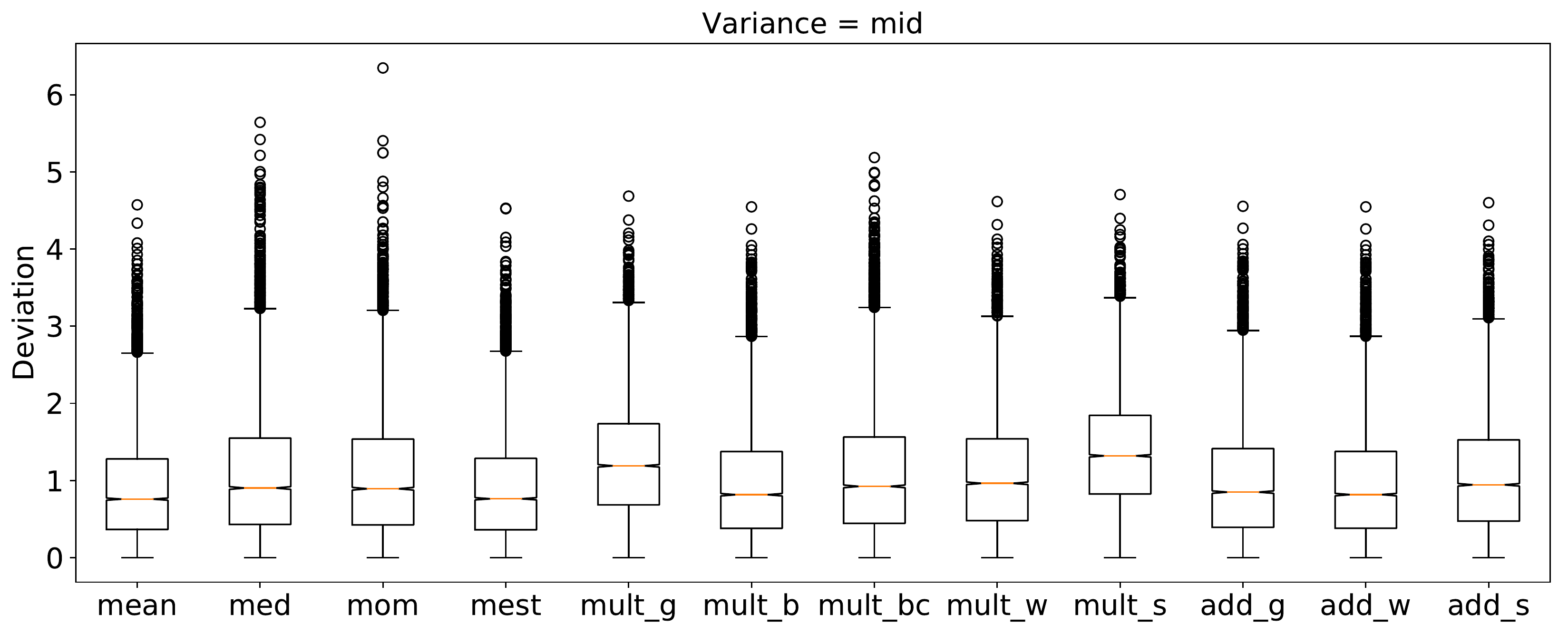}\includegraphics[width=0.49\textwidth]{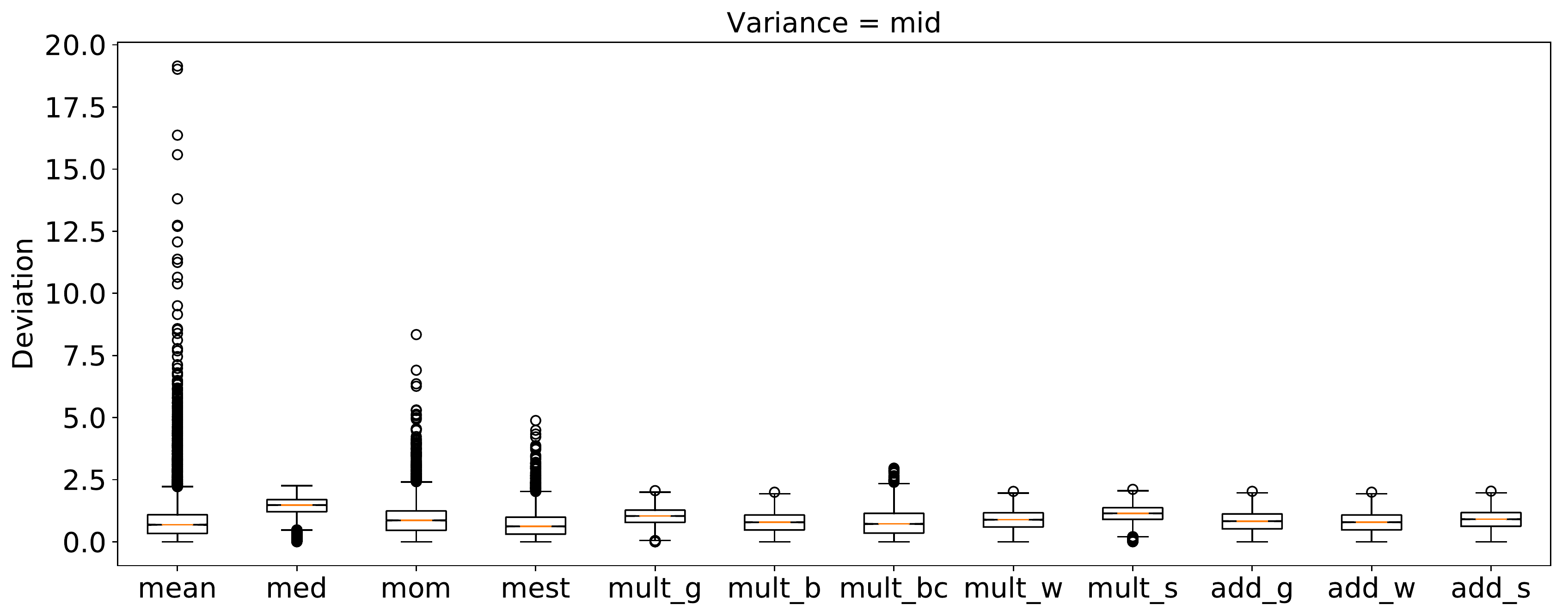}\\
\includegraphics[width=0.49\textwidth]{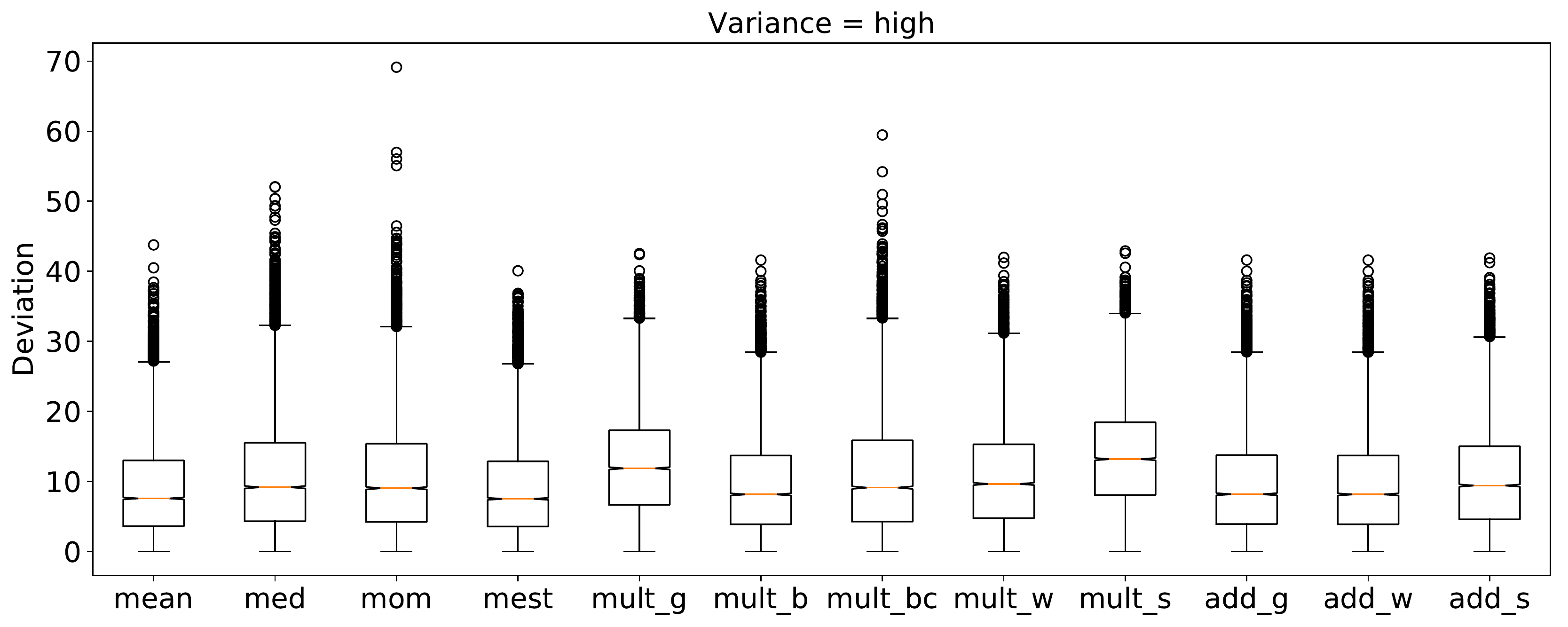}\includegraphics[width=0.49\textwidth]{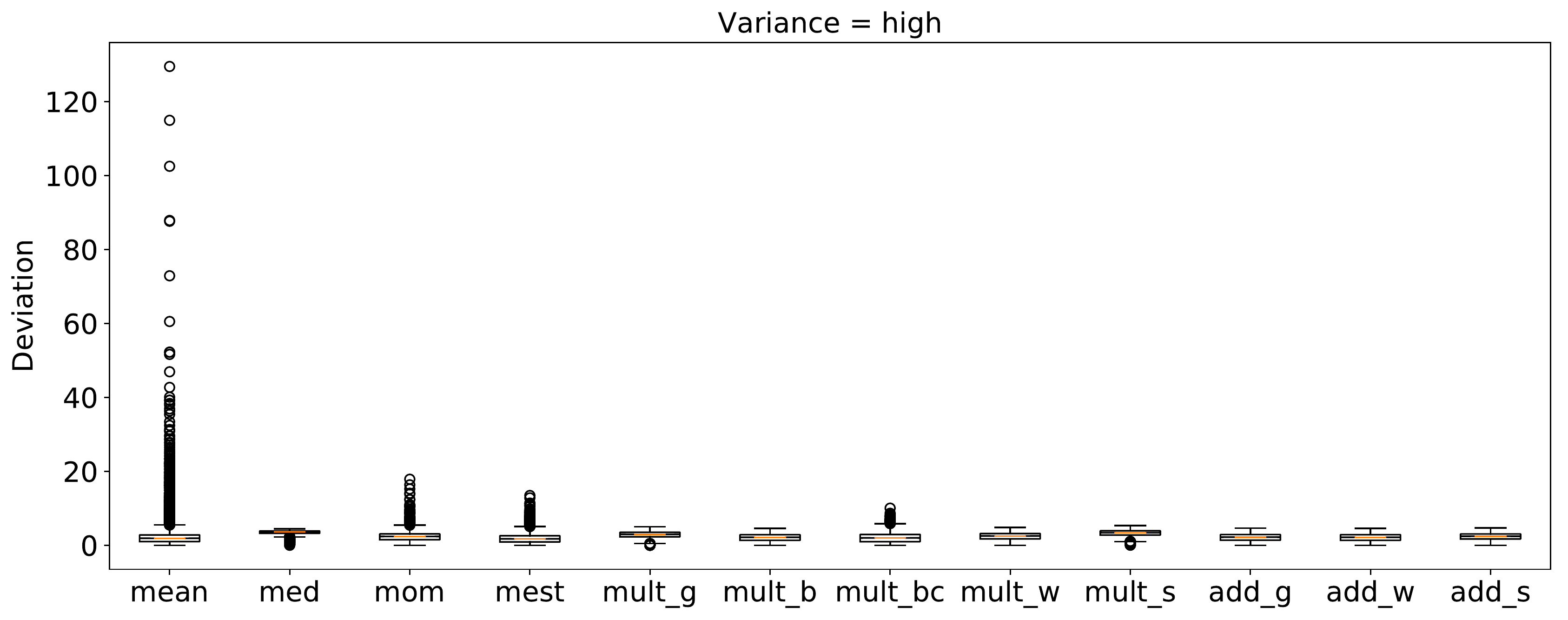}\\
\caption{Boxplots of deviations $|\xhat - \exx_{\ddist}X|$ over all trials. Sample size is $n=20$, mean to standard deviation ratio is $r(X)=1.0$. Left column: Normal data. Right column: log-Normal data. The rows correspond to low-high variance levels.}
\label{fig:box_deviations}
\end{figure}

In Figures \ref{fig:hist_deviations}--\ref{fig:box_deviations}, we look at the distribution of the absolute deviations observed over all independent trials. In the Normal case (left column of Figure \ref{fig:box_deviations}), the differences are subtle. All cases of $\xm$ and $\xa$ have deviations close to that of the sample mean; the bias towards zero is visible (noting $r(X)=1.0$, so $\exx X > 0$ here), but small, particularly in the case of the Bernoulli case of $\xm$, and all cases of $\xa$. On the other hand, in the log-Normal case, the strong bias of the median and the heavy-tailed deviations of the mean become salient, especially in the histograms of Figure \ref{fig:hist_deviations}. The concentration of various instances of $\xm$ and $\xa$ around non-zero deviation levels is indicative of the bias towards zero, though this trait appears much more clearly in the multiplicative Gaussian, Weibull, and Student-t cases than it does in any of the additive cases. This reinforces the observations made in section \ref{sec:empirical_meanSD}, where we looked at the average over all trials, rather than the entire distribution. Indeed, not only are the average levels similar, but the shape of the distribution of the \texttt{add\_*} methods is very close to that of \texttt{mult\_b}; this is lucid in the box plots for mid- to high-variance levels. On each histogram, we have indicated the maximum deviation value by a black dashed line, and for the non-classical estimators have also indicated the tightest known deviation bounds under just finite noise in matching color. While still somewhat loose, the $(1-2\delta)$ quantile for each distribution appears well below the derived bounds, providing a useful sanity check for both the new bounds and the computational formulas obtained in section \ref{sec:theory}. Among all the robust estimators, the centered estimator \texttt{mult\_bc} is clearly distinct, with an asymmetric distribution concentrated much closer to zero, indicating the smaller bias that we would expect due to the location correction being carried out (cf.~section \ref{sec:bydistro_bernoulli}).

\section{Conclusions and discussion}\label{sec:conclusions}

In this paper, we introduced and investigated a new class of mean estimators whose deviations enjoy exponential tails assuming only that the underlying data distribution has finite variance. These deviation bounds are sub-optimal in the sense that since the second moment controls the bounds, rather than variance, the guarantees will always be weaker than those of sub-Gaussian estimators as studied by \citet{devroye2016a}. What we lose in sharpness of guarantees, we gain in computational tractability, as the estimators introduced here can be computed in closed form depending on just distribution functions of known parametric families. Empirical tests illustrated the distributional robustness predicted by the theory, and the sensitivity to large values of the mean was shown to be easily mitigated using a simple sample-splitting strategy.

The function $\trunc$ used here is not particularly special in and of itself. Analogous theoretical results could be proved for estimators using a Huber-type influence function $\trunc$, by making a slight modification to the functions in (\ref{eqn:trunc_keyprop}) used to bound $\trunc$ from above and below. The form of computational results would naturally change, but the overall picture is the same as we have described above.

As mentioned in the introduction, mean estimators which perform well under weak assumptions on the underlying distribution plays an important role in developing machine learning algorithms with performance guarantees that hold for a wide variety of data. While the scaling strategy depends in the ideal case on partial information about the underlying distribution (known bounds on the second moments), in the case of M-estimators, cheap empirical estimates of the ideal scale have been shown to work well in practice \citep{holland2017a}. The most obvious direct application is to replace the core feedback mechanism is empirical risk minimization algorithms with a robustified objective, in the fashion of \citet{brownlees2015a}, which used Catoni-type M-estimators, which provide pointwise sub-Gaussian estimates of the risk, but yield an objective which is defined implicitly and challenging to minimize. Replacing this with the estimators discussed here could alleviate gaps between what can be achieved in practice and what is guaranteed in theory. Another natural application of interest is to gradient-based learning algorithms, which ignore the loss values and instead only try to approximate the gradient of the true risk function. Recent work in the machine learning community has looked at using median-of-means and M-estimators for constructing robust estimates of the risk gradient \citep{chen2017a,holland2019c}, but these iterative procedures can become expensive when the underlying dimension is high. In contrast, the estimators introduced here could replace existing estimators as a scalable strategy for robust learning algorithms in a high-dimensional setting. In particular, the Bernoulli-type estimator introduced in section \ref{sec:bydistro_bernoulli} is both theoretically and computationally appealing, and a natural candidate to accelerate recent robust gradient descent algorithms.

\appendix

\section{Technical appendix}\label{sec:tech}

\subsection{Proofs of results in main text}\label{sec:tech_proofs}

\begin{proof}[Proof of Proposition \ref{prop:bernoulli_all}]
The new form follows simply from direct computation:
\begin{align}
\nonumber
\xm & = \frac{s}{n} \sum_{i=1}^{n} \exx_{\epsilon_{i}} \trunc\left(\frac{X_{i}\epsilon_{i}}{s}\right)\\
\nonumber
& = \frac{s}{n} \sum_{i=1}^{n} \left( \theta\trunc\left(\frac{X_{i}}{s}\right) + (1-\theta)\trunc(0) \right)\\
\label{eqn:bernoulli_all_xhat_form}
& = \frac{\theta s}{n} \sum_{i=1}^{n} \trunc\left(\frac{X_{i}}{s}\right).
\end{align}
To obtain the deviation upper bound, we make use of the upper bound (\ref{eqn:xm_KL_bound}). The first term on the right-hand side of this bound is easily dealt with, observing
\begin{align*}
\int \left( \frac{\epsilon \exx_{\ddist}X}{s} + \frac{\epsilon^{2}\exx_{\ddist}X^{2}}{2s^{2}} \right) \, d\ndist(\epsilon) = \theta \left( \frac{\exx_{\ddist}X}{s} + \frac{\exx_{\ddist}X^{2}}{2s^{2}} \right)
\end{align*}
where we have used the fact that $\exx \epsilon^{2} = \exx \epsilon = \theta$ in the Bernoulli case. Furthermore, in the case of $\ndist_{\theta} = \text{Bernoulli}(\theta)$ and $\prior = \text{Bernoulli}(1/2)$, the relative entropy is computed as
\begin{align*}
\KL(\ndist_{\theta};\prior) = \theta \log(2\theta) + (1-\theta)\log(2(1-\theta)).
\end{align*}
Using these two equations to evaluate the right-hand side of (\ref{eqn:xm_KL_bound}), we obtain an upper bound on $\xm/\theta - \exx_{\ddist}X$, on an event of probability at least $1-\delta$ over the sample.

It remains to obtain a corresponding lower bound on $\xm/\theta - \exx_{\ddist}X$, or equivalently upper bounds on $(-1)\xm/\theta + \exx_{\ddist}X$. To do so, consider analogous settings of Bernoulli $\prior$ and $\ddist$, but this time on the domain $\{-1,0\}$, with $\ndist\{-1\}=\theta$ and $\prior\{-1\}=1/2$. Using (\ref{eqn:trunc_keyprop}) and Lemma \ref{lem:cat17type_KLbound} again, we have
\begin{align*}
\left(\frac{-\theta}{s}\right)\xhat & \leq \frac{-\theta \exx_{\ddist}X}{s} + \frac{\theta \exx_{\ddist}X^{2}}{2s^{2}} + \frac{1}{n} \left(\theta \log(2\theta) + (1-\theta)\log(2(1-\theta)) + \log(\delta^{-1})\right)
\end{align*}
where we note $\exx_{\ddist}\epsilon = -\theta$ and  $\exx_{\ddist}\epsilon^{2} = \exx_{\ddist}|\epsilon| = \theta$. This yields a high-probability lower bound in the desired form. Note however that since we have changed the prior distribution from the case where we proved the upper bound, we must take a union over the two events, yielding high-probability two-sided bounds on a $1-2\delta$ event.
\end{proof}

\begin{proof}[Proof of Proposition \ref{prop:normal_devbd}]
Starting with the additive case, evaluating the first term in the right-hand side of (\ref{eqn:xa_KL_bound}) and using the $\exx_{\ndist}\epsilon = 0$ assumption, we have that
\begin{align*}
\int \left( \frac{\exx_{\ddist}X + \epsilon}{s} + \frac{\exx_{\ddist}(X+\epsilon)^{2}}{2s^{2}} \right) \, d\ndist(\epsilon) & = \frac{\exx_{\ddist}X}{s} + \frac{1}{2s^{2}}\left(\exx_{\ndist}\epsilon^{2} + \exx_{\ddist}X^{2} \right)\\
& = \frac{\exx_{\ddist}X}{s} + \frac{1}{2\beta s^{2}} + \frac{\exx_{\ddist}X^{2}}{2s^{2}}.
\end{align*}
Using the fact that for $\ndist = \text{Normal}(0,\beta^{-1})$ and $\prior = \text{Normal}(1,\beta^{-1})$, the relative entropy takes the form $\KL(\ndist;\prior) = \beta/2$, using Lemma \ref{lem:cat17type_KLbound} we thus obtain an upper bound
\begin{align*}
\xa - \exx_{\ddist}X \leq \frac{1}{2\beta s} + \frac{\exx_{\ddist}X^{2}}{2s} + \frac{s}{n}\left(\frac{\beta}{2} + \log(\delta^{-1})\right)
\end{align*}
with at least probability $1-\delta$ in the draw of the data sample. Optimizing this bound in terms of $\beta > 0$ yields
\begin{align*}
\beta^{2} = \frac{n}{s^{2}},
\end{align*}
and with respect to $s>0$ yields
\begin{align*}
s^{2} = n \left(\frac{1}{2\beta} + \frac{\exx_{\ddist}X^{2}}{2}\right)\left(\frac{\beta}{2} + \log(\delta^{-1})\right)^{-1}.
\end{align*}
Plugging the former into the latter, we obtain a scale setting of
\begin{align*}
s^{2} = \frac{n\exx_{\ddist}X^{2}}{2\log(\delta^{-1})}.
\end{align*}
Finally plugging this into the upper bound we obtain
\begin{align*}
\xa - \exx_{\ddist}X \leq \sqrt{\frac{2\exx_{\ddist}X^{2}\log(\delta^{-1})}{n}} + \frac{1}{\sqrt{n}}
\end{align*}
on the same high-probability good event. To obtain a lower bound on $\xa - \exx_{\ddist}X$, equivalently an upper bound on $-\xa + \exx_{\ddist}X$, is straightforward. First, note that $-\trunc(u) \leq \log(1-u+u^{2}/2)$ via (\ref{eqn:trunc_keyprop}). From Lemma \ref{lem:cat17type_KLbound}, setting $f(x,\epsilon) = -\trunc((x+\epsilon)/2)$, we obtain bounds of the form
\begin{align*}
-\xa \leq -\exx_{\ddist}X + \frac{1}{2\beta s} + \frac{\exx_{\ddist}X^{2}}{2s} + \frac{s}{n}\left(\frac{\beta}{2} + \log(\delta^{-1})\right)
\end{align*}
which can then be optimized in $\beta$ and $s$ just as above, yielding a bound
\begin{align*}
-\xa+\exx_{\ddist}X \leq \sqrt{\frac{2\exx_{\ddist}X^{2}\log(\delta^{-1})}{n}} + \frac{1}{\sqrt{n}}
\end{align*}
on an event of probability at least $1-\delta$. Unfortunately, since we have modified our choice of $f$ used in Lemma \ref{lem:cat17type_KLbound}, the two $1-\delta$ events considered above need not coincide, and so using a union bound, we have two-sided bounds with probability at least $1-2\delta$ of the form
\begin{align*}
|\xa-\exx_{\ddist}X| \leq \sqrt{\frac{2\exx_{\ddist}X^{2}\log(\delta^{-1})}{n}} + \frac{1}{\sqrt{n}},
\end{align*}
which is the desired result for $\xa$.

For the multiplicative case, the proof is essentially analogous, except using equation (\ref{eqn:xm_KL_bound}) instead of (\ref{eqn:xa_KL_bound}), and we can obtain two-sided bounds without taking a union bound. A full proof appears in \citet{holland2019a}.
\end{proof}

\begin{proof}[Proof of Proposition \ref{prop:weibull_comp}]
When $u \leq 0$, since the event $\{W \leq 0\}$ has zero measure, by monotonicity we have $\exx W^{k} I\{W \leq u\} = 0$. Thus, it remains to consider the case where $u > 0$, which is assumed in all the following computations.

Using integration by parts, we have
\begin{align}\label{eqn:weibull_comp_W1prep}
\exx W I\{ W \leq u \} = \int_{0}^{u} \exp\left(-\left(\frac{t}{\sigma}\right)^{k}\right) \, dt - u \exp\left(-\left(\frac{u}{\sigma}\right)^{k}\right).
\end{align}
Examining the first term on the right-hand side of (\ref{eqn:weibull_comp_W1prep}), using substitution we have
\begin{align*}
\int_{0}^{u} \exp\left(-\left(\frac{t}{\sigma}\right)^{k}\right) \, dt & = \int_{0}^{u^{\prime}} e^{-t} \left(\frac{\sigma}{k}\right) t^{\frac{1}{k}-1} \, dt\\
& = \frac{\sigma}{k} \Gamma\left(\frac{1}{k};\left(\frac{u}{\sigma}\right)^{k}\right)
\end{align*}
where $u^{\prime} = (u/\sigma)^{k}$, and $\Gamma(u;v)$ is the unnormalized partial Gamma function. It follows that
\begin{align*}
\exx W I\{ W \leq u \} = \frac{\sigma}{k} \Gamma\left(\frac{1}{k};\left(\frac{u}{\sigma}\right)^{k}\right) - u \exp\left(-\left(\frac{u}{\sigma}\right)^{k}\right).
\end{align*}

Next, again by parts, we have
\begin{align}\label{eqn:weibull_comp_W2prep}
\exx W^{2} I\{ W \leq u \} = 2 \int_{0}^{u} t \exp\left(-\left(\frac{t}{\sigma}\right)^{k}\right) \, dt - u^{2} \exp\left(-\left(\frac{u}{\sigma}\right)^{k}\right).
\end{align}
To evaluate the first term on the right-hand side of (\ref{eqn:weibull_comp_W2prep}), again we use substitution, yielding
\begin{align*}
\int_{0}^{u} t \exp\left(-\left(\frac{t}{\sigma}\right)^{k}\right) \, dt & = \int_{0}^{u^{\prime}} \left(\sigma t^{\frac{1}{k}}\right) e^{-t} \left(\frac{\sigma}{k}\right) t^{\frac{1}{k}-1} \, dt\\
& = \frac{\sigma^{2}}{k} \Gamma\left(\frac{2}{k}; \left(\frac{u}{\sigma}\right)^{k}\right)
\end{align*}
It thus follows that
\begin{align*}
\exx W^{2} I\{ W \leq u \} = \frac{2\sigma^{2}}{k} \Gamma\left(\frac{2}{k}; \left(\frac{u}{\sigma}\right)^{k}\right) - u^{2} \exp\left(-\left(\frac{u}{\sigma}\right)^{k}\right).
\end{align*}

Finally, again by parts, we have
\begin{align}\label{eqn:weibull_comp_W3prep}
\exx W^{3} I\{ W \leq u \} = 3 \int_{0}^{u} t^{2} \exp\left(-\left(\frac{t}{\sigma}\right)^{k}\right) \, dt - u^{3} \exp\left(-\left(\frac{u}{\sigma}\right)^{k}\right).
\end{align}
To evaluate the first term on the right-hand side of (\ref{eqn:weibull_comp_W3prep}), again using substitution, we have
\begin{align*}
\int_{0}^{u} t^{2} \exp\left(-\left(\frac{t}{\sigma}\right)^{k}\right) \, dt & = \int_{0}^{u^{\prime}} \left(\sigma t^{\frac{1}{k}}\right)^{2} e^{-t} \frac{\sigma}{k} t^{\frac{1}{k}-1} \, dt\\
& = \frac{\sigma^{3}}{k} \Gamma\left(\frac{3}{k};\left(\frac{u}{\sigma}\right)^{k}\right).
\end{align*}
It follows that
\begin{align*}
\exx W^{3} I\{ W \leq u \} = 3 \frac{\sigma^{3}}{k} \Gamma\left(\frac{3}{k};\left(\frac{u}{\sigma}\right)^{k}\right) - u^{3} \exp\left(-\left(\frac{u}{\sigma}\right)^{k}\right).
\end{align*}
With the above forms in place, the desired result follows immediately.
\end{proof}

\begin{proof}[Proof of Proposition \ref{prop:weibull_devbd}]
We have $W \sim \text{Weibull}(k,\sigma)$. Recall that the mean and variance for a general Weibull distribution have the forms
\begin{align*}
\exx W & = \sigma \Gamma(1+1/k)\\
\vaa W & = \sigma^{2}\left(\Gamma(1+2/k)-\Gamma^{2}(1+1/k)\right).
\end{align*}

We start with the additive case. In this setting, since $k=2$ by assumption and thus $\exx W = \sigma\Gamma(3/2)$, we have centered noise $\exx_{\ndist}\epsilon = 0$ as desired; note also that $\Gamma(3/2) = \sqrt{\pi}/2$. First of all, because of this centering, we have $\exx_{\ndist} \epsilon^{2} = \vaa W = \sigma^{2}(1-\Gamma^{2}(3/2))$. In arguments analogous to Propositions \ref{prop:bernoulli_all} and \ref{prop:normal_devbd}, on the high-probability event we have
\begin{align*}
\xa - \exx_{\ddist}X \leq \frac{\exx_{\ddist}X^{2}+\sigma^{2}(1-\pi/4)}{2s} + \frac{s}{n}\left(\KL(\ndist;\prior)+\log(\delta^{-1})\right)
\end{align*}
and lower bounds follow just as in the proof of Proposition \ref{prop:normal_devbd}, yielding two-sided bounds on an event of $1-2\delta$ after taking a union bound. The relative entropy vanishes by assumption, and optimizing this bound with respect to $s>0$ yields the desired result for $\xa$.

Next we handle the multiplicative case. Setting $\ndist = \text{Weibull}(k,\sigma(k))$ with $\sigma(k) = (\Gamma(1+1/k))^{-1}$ ensures that $\exx_{\ndist}\epsilon = 1$. Using (\ref{eqn:xm_KL_bound}), we thus have
\begin{align*}
\xm - \exx_{\ddist}X \leq \frac{\Gamma(1+2/k)\exx_{\ddist}X^{2}}{\Gamma^{2}(1+1/k)\,2s} + \frac{s}{n}\left(\KL(\ndist;\prior)+\log(\delta^{-1})\right)
\end{align*}
on the high-probability event. Since the relative entropy between two Weibull distributions in the general case takes the form given by (\ref{eqn:weibull_KL}), this yields
\begin{align*}
\KL(\ndist;\prior) = \frac{1}{\Gamma^{k}(1+1/k)} + k\log\Gamma(1+1/k) - 1.
\end{align*}
Plugging in this value to the upper bound and optimizing with respect to $s>0$ yields the desired result. Lower bounds here work in the same way as the additive case shown above.
\end{proof}

\begin{proof}[Proof of Proposition \ref{prop:student_comp}]
To keep the computations from getting too cluttered, we establish some notation before starting with the proof. For positive integer $\dof$, write the normalization constant of the $\text{Student}(\dof)$ distribution as
\begin{align*}
A_{\dof} & \defeq \frac{\Gamma((\dof+1)/2)}{\sqrt{\dof\pi}\,\Gamma(\dof/2)}.
\end{align*}

To begin, the first quantity to be evaluated is
\begin{align*}
\exx W I\{W \leq u\} = \int_{-\infty}^{u} t \, A_{\dof} \left(1 + \frac{t^{2}}{\dof}\right)^{-(\dof+1)/2} \, dt.
\end{align*}
In seeking an anti-derivative, first observe that
\begin{align}\label{eqn:student_comp_E1_start}
\frac{d}{dt} \left(1+\frac{t^{2}}{\dof}\right)^{-(\dof+1)/2} = (-1)\frac{(\dof+1)}{\dof} t \left(1 + \frac{t^{2}}{\dof}\right)^{-(\dof+3)/2}.
\end{align}
Integrating the left-hand side of (\ref{eqn:student_comp_E1_start}) can be done using the fundamental theorem of calculus, and the right-hand side  almost has the desired form, save for the $\dof+3$ in the exponent, which does not match the form of the Student density. This can be easily dealt with using a substitution argument. With transformation $g(u) \defeq u \sqrt{\dof/(\dof+2)}$, by substitution we have
\begin{align}
\nonumber
\int_{-\infty}^{g(u)} (-1)\frac{(\dof+1)}{\dof} t \left(1 + \frac{t^{2}}{\dof}\right)^{-(\dof+3)/2} \, dt & = \int_{-\infty}^{u} (-1)\frac{(\dof+1)}{\dof} g(t) \left(1 + \frac{g(t)^{2}}{\dof}\right)^{-(\dof+3)/2} g^{\prime}(t) \, dt\\
\nonumber
& = \int_{-\infty}^{u} (-1)\frac{(\dof+1)}{(\dof+2)} t \left(1 + \frac{t^{2}}{\dof+2}\right)^{-(\dof+3)/2} \, dt\\
\label{eqn:student_comp_E1_sub}
& \defeq \int_{-\infty}^{u} t \, \widetilde{p}^{(1)}_{\dof+2}(t) \, dt.
\end{align}
Denoting by $p_{\dof}$ by the $\text{Student}(\dof)$ density function, note that with rescaling we have
\begin{align*}
(-1)\frac{(\dof+2)}{(\dof+1)} A_{\dof+2} \, \widetilde{p}^{(1)}_{\dof+2}(t) = p_{\dof+2}(t), \quad t \in \RR.
\end{align*}
It thus immediately follows that for $\dof > 1$, using (\ref{eqn:student_comp_E1_start}) and (\ref{eqn:student_comp_E1_sub}) we have
\begin{align}
\nonumber
\exx_{\dof+2} W I\{W \leq u \} & = (-1)\frac{(\dof+2)}{(\dof+1)} A_{\dof+2} \int_{-\infty}^{u} t \, \widetilde{p}^{(1)}_{\dof+2}(t) \, dt\\
\nonumber
& = (-1)\frac{(\dof+2)}{(\dof+1)} A_{\dof+2} \left[ \left(1 + \frac{t^{2}}{\dof}\right)^{-(\dof+1)/2} \right]^{g(u)}_{-\infty}\\
\label{eqn:student_comp_E1_final}
& = (-1)\frac{(\dof+2)}{(\dof+1)} A_{\dof+2} \left(1 + \frac{u^{2}}{\dof+2}\right)^{-(\dof+1)/2}.
\end{align}
Thus, for any $\dof > 3$, we can compute $\exx_{\dof} W I\{W \leq u\}$ by the formula given by (\ref{eqn:student_comp_E1_final}).

The second quantity of interest is
\begin{align*}
\exx W^{2} I\{ W \leq u \} = \int_{-\infty}^{u} t^{2} \, A_{\dof} \left(1 + \frac{t^{2}}{\dof+2}\right)^{-(\dof+1)/2} \, dt.
\end{align*}
In a similar fashion to the first-degree case, here observe that
\begin{align}\label{eqn:student_comp_E2_start}
\frac{d}{dt} \, t \left(1 + \frac{t^{2}}{\dof}\right)^{-(\dof+1)/2} = \left(1 + \frac{t^{2}}{\dof}\right)^{-(\dof+1)/2} - \frac{\dof+1}{\dof} t^{2} \left(1 + \frac{t^{2}}{\dof}\right)^{-(\dof+3)/2}.
\end{align}
Using the fundamental theorem of calculus again, along with the Student CDF, we can thus readily solve for the integral of the second term on the right-hand side. To relate this term to the desired quantity, once again we use substitution, with transformation $g(u) = u \sqrt{\dof/(\dof+2)}$, to obtain
\begin{align}
\nonumber
\int_{-\infty}^{g(u)} \frac{(\dof+1)}{\dof} \, t^{2} \left(1 + \frac{t^{2}}{\dof}\right)^{-(\dof+3)/2} \, dt & = \int_{-\infty}^{u} \frac{(\dof+1)}{\dof} \, g(t)^{2} \left(1 + \frac{g(t)^{2}}{\dof}\right)^{-(\dof+3)/2} g^{\prime}(t) \, dt\\
\nonumber
& = \int_{-\infty}^{u} \frac{\sqrt{\dof}(\dof+1)}{(\dof+2)^{3/2}} \, t^{2} \left(1 + \frac{t^{2}}{\dof+2}\right)^{-(\dof+3)/2} \, dt\\
\label{eqn:student_comp_E2_sub}
& \defeq \int_{-\infty}^{u} t^{2} \widetilde{p}^{(2)}_{\dof+2}(t) \, dt.
\end{align}
With proper rescaling, we have
\begin{align*}
\frac{(\dof+2)^{3/2}}{\sqrt{\dof}(\dof+1)} \, A_{\dof+2} \,  \widetilde{p}^{(2)}_{\dof+2}(t) = p_{\dof+2}(t), \quad t \in \RR.
\end{align*}
Thus using (\ref{eqn:student_comp_E2_start}) and (\ref{eqn:student_comp_E2_sub}), we can conclude that for any $\dof > 2$, we have
\begin{align}
\nonumber
\exx_{\dof+2} W^{2} I\{W \leq u \} & = \frac{(\dof+2)^{3/2}}{\sqrt{\dof}(\dof+1)} \, A_{\dof+2} \int_{-\infty}^{u} t^{2} \widetilde{p}^{(2)}_{\dof+2}(t) \, dt\\
\nonumber
& \hspace{-2.5cm} = \frac{(\dof+2)^{3/2}}{\sqrt{\dof}(\dof+1)} \, A_{\dof+2} \left( \int_{-\infty}^{g(u)}\left(1+\frac{t^{2}}{\dof}\right)^{-(\dof+1)/2}\,dt - \left[ t \left(1 + \frac{t^{2}}{\dof}\right)^{-(\dof+1)/2} \right]^{g(u)}_{-\infty} \right)\\
\label{eqn:student_comp_E2_final}
& \hspace{-2.5cm} = \frac{(\dof+2)^{3/2}}{\sqrt{\dof}(\dof+1)} \, \frac{A_{\dof+2}}{A_{\dof}} \exx_{\dof} I\left\{ W^{\prime} \leq u \sqrt{\frac{\dof}{\dof+2}} \right\} - A_{\dof+2} \frac{(\dof+2)}{(\dof+1)} \, u \left(1+\frac{u^{2}}{\dof+2}\right)^{-(\dof+1)/2}
\end{align}
where $W^{\prime}$ is an independent copy of $W$, here with distribution $\text{Student}(\dof)$. For any $\dof > 4$ then, using the formula given by (\ref{eqn:student_comp_E2_final}).

The final quantity of interest is handled in an analogous way. We seek
\begin{align*}
\exx W^{3} I\{ W \leq u \} = \int_{-\infty}^{u} t^{3} \, A_{\dof} \left(1 + \frac{t^{2}}{\dof+2}\right)^{-(\dof+1)/2} \, dt
\end{align*}
and to start us off compute
\begin{align}\label{eqn:student_comp_E3_start}
\frac{d}{dt} \, t^{2} \left(1 + \frac{t^{2}}{\dof}\right)^{-(\dof+1)/2} = 2t \left(1+\frac{t^{2}}{\dof}\right)^{-(\dof+1)/2} - \frac{(\dof+1)}{\dof} t^{3} \left(1+\frac{t^{2}}{\dof}\right)^{-(\dof+3)/2}
\end{align}
and to relate the last term to the desired quantity, using substitution with transformation $g(u) = u \sqrt{\dof/(\dof+2)}$, we obtain
\begin{align}
\nonumber
\int_{-\infty}^{g(u)} \frac{(\dof+1)}{\dof} t^{3} \left(1+\frac{t^{2}}{\dof}\right)^{-(\dof+3)/2} \, dt & = \int_{-\infty}^{u} \frac{(\dof+1)}{\dof} g(t)^{3} \left(1+\frac{g(t)^{2}}{\dof}\right)^{-(\dof+3)/2} g^{\prime}(t) \, dt\\
\nonumber
& = \int_{-\infty}^{u} \frac{\dof(\dof+1)}{(\dof+2)^{2}} t^{3} \left(1+\frac{t^{2}}{\dof+2}\right)^{-(\dof+3)/2} \, dt\\
\label{eqn:student_comp_E3_sub}
& \defeq \int_{-\infty}^{u} t^{3} \widetilde{p}^{(3)}_{\dof+2}(t) \, dt.
\end{align}
Rescaling to obtain the desired density, we have
\begin{align*}
\frac{(\dof+2)^{2}}{\dof(\dof+1)} A_{\dof+2} \, \widetilde{p}^{(3)}_{\dof+2}(t) = p_{q+2}(t), \quad t \in \RR.
\end{align*}
With this in hand, using (\ref{eqn:student_comp_E3_start}) and (\ref{eqn:student_comp_E3_sub}), it follows for $\dof > 3$ that
\begin{align}
\nonumber
\exx_{\dof+2} W^{3} I\{W \leq u\} & = \frac{(\dof+2)^{2}}{\dof(\dof+1)} A_{\dof+2} \int_{-\infty}^{u} t^{3} \widetilde{p}^{(3)}_{\dof+2}(t) \, dt\\
\nonumber
& \hspace{-2.5cm} = \frac{(\dof+2)^{2}}{\dof(\dof+1)} A_{\dof+2} \left( \frac{2}{A_{\dof}} \exx_{\dof} W^{\prime} I\left\{ W^{\prime} \leq g(u) \right\} - \left[ t^{2} \left(1+\frac{t^{2}}{\dof}\right)^{-(\dof+1)/2} \right]^{g(u)}_{-\infty} \right)\\
\label{eqn:student_comp_E3_final}
& \hspace{-2.5cm} = \frac{2(\dof+2)^{2}}{\dof(\dof+1)} \frac{A_{\dof+2}}{A_{\dof}} \exx_{\dof} W^{\prime} I\left\{ W^{\prime} \leq u \sqrt{\frac{\dof}{\dof+2}} \right\} - \frac{(\dof+2)}{(\dof+1)} A_{\dof+2} \, u^{2} \left(1+\frac{u^{2}}{\dof+2}\right)^{-(\dof+1)/2}
\end{align}
where again $W^{\prime}$ is an independent copy of $W$ with distribution $\text{Student}(\dof)$. We may thus conclude that for any $\dof > 5$, using the formula given by (\ref{eqn:student_comp_E3_final}), we are able to compute the desired $\exx_{\dof} W^{3} I\{ W \leq u\}$.
\end{proof}

\begin{proof}[Proof of Proposition \ref{prop:student_devbd}]
Write $W \sim \text{Student}(\dof)$ for $\dof>1$, with shift parameter $\alpha$ as specified in the hypothesis. By symmetry we have that $\exx W = 0$. Assuming $\dof>2$, the second moment is $\exx W^{2} = \dof/(\dof-2)$. To obtain deviation bounds for $\xm$ and $\xa$, once again using (\ref{eqn:xm_KL_bound}) and (\ref{eqn:xa_KL_bound}) as in previous proofs, it follows that each of the following inequalities holds on an event of probability at least $1-\delta$:
\begin{align*}
\xm & \leq \alpha\,\exx_{\ddist}X + \frac{(\alpha^{2}(\dof-2)+\dof)\exx_{\ddist}X^{2}}{(\dof-2)2s} + \frac{s}{n}\left(\KL(\ndist;\prior)+\log(\delta^{-1})\right)\\
\xa & \leq \exx_{\ddist}X + \frac{\exx_{\ddist}X^{2}+\dof/(\dof-2)}{2s} + \frac{s}{n}\left(\KL(\ndist;\prior)+\log(\delta^{-1})\right).
\end{align*}
To deal with the relative entropy, we obtain an upper bound on it as follows. To start, consider the additive case, where we have a prior of $\text{Student}(\dof)$ shifted by $\alpha \in \RR$, and thus has a density function of $p_{\alpha}(u) = p(u+\alpha)$, where $p$ is the $\text{Student}(\dof)$ density (\ref{eqn:student_density}). Normalizing constants cancel, and the relative entropy takes the form
\begin{align*}
\KL(\ndist;\prior) & = \frac{(\dof+1)}{2}\int_{-\infty}^{\infty} \left( \log\left(1+\frac{(u+\alpha)^{2}}{\dof}\right) - \log\left(1+\frac{u^{2}}{\dof}\right) \right) p(u) \, du.
\end{align*}
Directly evaluating this function is challenging, but a simple upper bound will suit our needs. The first term can be controlled as follows. Using the sub-additivity of the concave function $f(u)=\log(1+u)$, we have
\begin{align*}
\exx \log\left(1+\frac{(W+\alpha)^{2}}{\dof}\right) & = \exx\log\left(1+\frac{W^{2}+\alpha^{2}+2\alpha W}{\dof}\right)\\
& \leq \exx\log\left(1+\frac{W^{2}+\alpha^{2}+2|\alpha W|}{\dof}\right)\\
& \leq \exx\log\left(1+\frac{W^{2}}{\dof}\right) + \log\left(1+\frac{\alpha^{2}}{\dof}\right) + \exx\log\left(1+\frac{2|\alpha W|}{\dof}\right).
\end{align*}
The first term here can be evaluated directly in the form
\begin{align}\label{eqn:student_devbd_KL_tsquared}
\exx\log\left(1+\frac{W^{2}}{\dof}\right) = \Psi\left(\frac{\dof+1}{2}\right) - \Psi\left(\frac{\dof}{2}\right),
\end{align}
where $\Psi(u) = (d/du)\log\Gamma(u)$ is the digamma function. This is a well-known classical fact. To see that it is valid, first recall that a squared Student-t random variable with $\dof$ degrees of freedom has the same distribution as the ratio of two independent chi-squared random variables \citep{stuart1994KendallVol1}, namely
\begin{align*}
W^{2} = \frac{\chi^{2}_{1}}{(\chi^{2}_{\dof}/\dof)}
\end{align*}
where equality here refers to equality in distribution. Equality in distribution implies equality in mean, and thus
\begin{align*}
\exx\log\left(1+\frac{W^{2}}{\dof}\right) & = \exx\log\left(1+\frac{\chi^{2}_{1}}{\chi^{2}_{\dof}}\right)\\
& = \exx\log\left(\frac{\chi^{2}_{\dof}+\chi^{2}_{1}}{\chi^{2}_{\dof}}\right)\\
& = \exx\log\left(\chi^{2}_{\dof+1}\right)-\exx\log\left(\chi^{2}_{\dof}\right)
\end{align*}
where we use the fact that Chi-squared random variables with $m$ degrees of freedom are equivalent to the sum of $m$ squared Normal random variables. Another important relation is that $\chi^{2}_{\dof} \sim \text{Gamma}(\dof/2,1/2)$, namely Chi-squared has the same distribution as a Gamma random variable with shape $\dof/2$ and rate $1/2$. This is important because log-transformations of the Gamma distribution are well-understood. In particular, if $G_{\beta} \sim \text{Gamma}(\kappa,\beta)$ for arbitrary shape $\kappa$, then in the special case of $\beta=1$, we have
\begin{align*}
\exx \log(G_{1}) = \left.\frac{d}{du}\log\Gamma(u)\right|_{u=\kappa} = \Psi(\kappa),
\end{align*}
a fact that dates back to at least \citet{johnson1949a} (see also \citet[Ch.~6 exercises]{stuart1994KendallVol1}). Using integration by substitution, we have that
\begin{align*}
\exx\log(G_{1/2}) = \log(2) + \exx\log(G_{1}),
\end{align*}
which means that using the relation of Chi-squared to Gamma, we have
\begin{align*}
\exx\log\left(1+\frac{W^{2}}{\dof}\right) & = \exx\log\left(\chi^{2}_{\dof+1}\right)-\exx\log\left(\chi^{2}_{\dof}\right)\\
& = \left(\Psi\left(\frac{\dof+1}{2}\right) + \log(2)\right) - \left(\Psi\left(\frac{\dof}{2}\right) + \log(2)\right)\\
& = \Psi\left(\frac{\dof+1}{2}\right)-\Psi\left(\frac{\dof}{2}\right)
\end{align*}
which is the desired form (\ref{eqn:student_devbd_KL_tsquared}). The remaining term to control is $\exx\log\left(1+2|\alpha W|/\dof\right)$. Using the inequality $\log(1+u) \leq u$, we have that
\begin{align*}
\exx\log\left(1+\frac{2|\alpha W|}{\dof}\right) \leq \frac{2|\alpha| \exx|W|}{\dof} = \frac{4|\alpha|\Gamma((\dof+1)/2)}{\sqrt{\dof\pi}\Gamma(\dof/2)(\dof-1)}
\end{align*}
which follows from a straightforward derivation of the mean of a folded Student-t. We thus have a straightforward upper bound on the relative entropy, taking the form
\begin{align*}
\KL(\ndist;\prior) & = \frac{(\dof+1)}{2}\left( \exx\log\left(1+\frac{(W+\alpha)^{2}}{\dof}\right) - \exx\log\left(1+\frac{W^{2}}{\dof}\right) \right)\\
& \leq \frac{(\dof+1)}{2}\left[\left( \Psi_{1}-\Psi_{0} + \log\left(1+\frac{\alpha^{2}}{\dof}\right) + \frac{4|\alpha|\Gamma((\dof+1)/2)}{\sqrt{\dof\pi}\,\Gamma(\dof/2)(\dof-1)} \right) - \left( \Psi_{1}-\Psi_{0} \right)\right]\\
& = \frac{(\dof+1)}{2}\left( \log\left(1+\frac{\alpha^{2}}{\dof}\right) + \frac{4|\alpha|\Gamma((\dof+1)/2)}{\sqrt{\dof\pi}\,\Gamma(\dof/2)(\dof-1)} \right)
\end{align*}
where for readability, we have used the notation
\begin{align*}
\Psi_{1} \defeq \Psi\left(\frac{\dof+1}{2}\right), \quad \Psi_{0} \defeq \Psi\left(\frac{\dof}{2}\right).
\end{align*}
This gives us an upper bound on the relative entropy in the additive case, where we have $\prior = \text{Student}(\dof)-\alpha, \ndist = \text{Student}(\dof)$. For the multiplicative case, where we assume $\prior = \text{Student}(\dof), \ndist = \text{Student}(\dof)+\alpha$, the exact same bound on the relative entropy holds. To see this, note that in the latter setting, using substitution, we have
\begin{align*}
\KL(\ndist;\prior) & = \frac{(q+1)}{2} \int_{-\infty}^{\infty} \left(\log\left(1+\frac{u^{2}}{\dof}\right) - \log\left(1+\frac{(u-\alpha)^{2}}{\dof}\right)\right) p(u-\alpha) \, du\\
& = \frac{(q+1)}{2} \int_{-\infty}^{\infty} \left(\log\left(1+\frac{(u+\alpha)^{2}}{\dof}\right) - \log\left(1+\frac{u^{2}}{\dof}\right)\right) p(u) \, du
\end{align*}
which is precisely the quantity we just bounded above. Plugging in this value as an upper bound for the relative entropy, and optimizing this upper bound with respect to $s$ yields the desired results. Lower bounds on $\xa-\exx_{\ddist}X$ and $\xm-\exx_{\ddist}X$ are obtained using an analogous argument to that seen in Propositions \ref{prop:normal_devbd} and \ref{prop:weibull_devbd}, yielding the desired two-sided deviation bounds with probability at least $1-2\delta$.
\end{proof}

{\small
\bibliographystyle{apalike}
\bibliography{1dim.bib}
}

\end{document}